\documentclass[12pt]{article}

\usepackage{amssymb}%
\usepackage{amsmath}%
\usepackage{amsthm}%
\usepackage{fullpage}

\usepackage{xcolor}%

\allowdisplaybreaks%

{\theoremstyle{plain}%
 \newtheorem{theorem}{Theorem}
 
 \newtheorem{lemma}{Lemma}%
}
{\theoremstyle{remark}
\newtheorem{remark}{Remark}
}
{\theoremstyle{definition}

\newtheorem{example}{Example}
}

\begin{document}

\begin{center}
{\Large Analytic dilogarithm identities}

 \ 

Cetin Hakimoglu-Brown

 \ 

\end{center}

\begin{abstract}
 We introduce dilogarithm identities through a beta integral-based technique that we apply to provide analytic proofs of previously 
 conjectured dilogarithm relations, solving open problems given by both Bytsko and Campbell, and that we further apply to construct 
 and prove new ladder relations with quartic and sextic bases. We also apply our method to introduce and prove two-term 
 $\operatorname{Li}_2$-relations and ladder-like identities with arguments in algebraic number fields such as $\mathbb{Q}(\sqrt{2})$, 
 $\mathbb{Q}(\sqrt{3})$, $\mathbb{Q}(\sqrt{5})$, and $\mathbb{Q}(\sqrt{-7})$. Moreover, single-term dilogarithm evaluations 
 are introduced and derived. 
\end{abstract}


\vspace{0.1in}


\section{Introduction}
 Higher logarithm functions are of fundamental importance in many areas of both physics and mathematics, and the dilogarithm function 
 $\operatorname{Li}_{2}$ may be seen as the most basic and fundamental function of this form. Polylogarithms arise in deep areas of 
 number theory and quantum theory, and the study of higher logarithm functions in its own right provides one of the core topics 
 concerning special functions. Zagier \cite{Zagier2007} built on the pioneering work of Nahm et al.\ 
 \cite{Nahm1995,Nahm2007,Nahm1994,NahmRecknagelTerhoeven1993} in ways that relate the use of higher $K$-groups and higher 
 Bloch groups to the behavior of the $\operatorname{Li}_{2}$-function, and the seminal nature of the monographs by Lewin on the 
 dilogarithm and polylogarithm functions \cite{Lewin1958,Lewin1981,Lewin1991text} further motivates the techniques and 
 $ \operatorname{Li}_{2}$-identities introduced in this paper, in which we solve open problems given by both Bytsko 
 \cite{Bytsko1999} and Campbell \cite{Campbell2025}. 

 The problem of evaluating inverse binomial series in terms of polylogarithms and related special functions often leads to rich areas of 
 research within computational and experimental mathematics \cite[\S1.7]{BorweinBaileyGirgensohn2004}. For hypergeometric series 
 of this form with half-integer parameters, the evaluation of such series is fairly well understood, and our work is inspired by a known 
 evaluation of this form introduced by D'Aurizio and subsequently explored by D'Aurizio and Di Trani \cite{DAurizioDiTrani2018}. By 
 letting the \emph{Pochhammer symbol} be such that $(x)_{n} = x(x+1) \cdots (x + n - 1)$ for $n \in \mathbb{N}_{0}$, and by letting 
 \emph{generalized hypergeometric series} be denoted and defined so that 
\begin{equation*}
 {}_{p}F_{q}\!\!\left[ \begin{matrix} a_{1}, a_{2}, \ldots, a_{p} \vspace{1mm}\\ b_{1}, b_{2}, \ldots, b_{q} \end{matrix} \ \Bigg| \ x 
 \right] = \sum_{n=0}^{\infty} \frac{ \left( a_{1} \right)_{n} \left( a_{2} \right)_{n} \cdots \left( a_{p} \right)_{n} }{ \left( b_{1} \right)_{n} 
 \left( b_{2} \right)_{n} \cdots \left( b_{q} \right)_{n} } \frac{x^{n}}{n!}, 
\end{equation*}
 the above referenced D'Aurizio evaluation is such that 
\begin{equation*}
 {}_{4}F_{3}\!\!\left[ \begin{matrix} 1, 1, 1, \frac{3}{2} \vspace{1mm}\\ \frac{5}{2}, \frac{5}{2}, \frac{5}{2} \end{matrix} \ \Bigg| \ 1 \right] = 
 27 \sum_{n\geq 0}\frac{16^n}{(2n+3)^3 (2n+1)^2 \binom{2n}{n}^2} = {\frac{27}{2}\left(7\,\zeta(3)+(3 - 
 2 G)\,\pi-12\right)} 
\end{equation*}
 for Ap\'{e}ry's constant $\zeta(3) = \frac{1}{1^3} + \frac{1}{2^3} + \cdots$ and Catalan's constant $G = \frac{1}{1^2} - \frac{1}{3^2} + 
 \frac{1}{5^2} - \cdots$. Relatively little is known when it comes to the evaluation of inverse binomial series that have third-integer 
 or quarter-integer or sixth-integer values as lower parameters, in contrast to the half-integer parameter values involved in D'Aurizio's 
 ${}_{4}F_{3}(1)$-series, but evaluations for ${}_{4}F_{3}$-series of the form 
\begin{equation}\label{4F3thirdinteger} 
 {}_{4}F_{3}\!\!\left[ \begin{matrix} 1, 1, 1, \frac{3}{2} \vspace{1mm}\\ 
 \frac{4}{3}, \frac{5}{3}, 2 \end{matrix} \ \Bigg| \ z \right] 
 = \frac{4}{9 z} \sum_{n=1}^{\infty} \left( \frac{27z}{4} \right)^{n} \frac{1}{n^2 \binom{3 n}{n}} 
\end{equation}
 and for equivalent or related series have been considered by a number of authors 
 \cite{AdegokeFrontczakGoy2024,Batir2005,CampbellLevrie2024,Chu2022,ChudnovskyChudnovsky1998,DAurizioDiTrani2018}, with 
 \eqref{4F3thirdinteger} admitting an analytic evaluation. Similar inverse binomial series related to the quarter-integer and sixth-integer 
 cases have been investigated, respectively, by Rathie and Shpot \cite{RathieShpot2025} and by Sun 
 et al.\ \cite{Sun2024,SunZhoupreprint}. 
 
 The evaluation techniques introduced in this paper are related to those in the above references on inverse binomial series, but may be 
 seen as providing a more comprehensive treatment of dilogarithm identities. Our beta integral-based series evaluation methods build 
 on our past work \cite{Hakimoglupreprint} and appear to be original. For background 
 material on computational approaches 
 toward the evaluation of series, we refer to the monograph due to Sofo \cite{SOFO2003}. 

\subsection{Preliminaries}
 The \emph{$\Gamma$-function} is among the most fundamental and ubiquitous special functions and may be defined via the Euler 
 integral $\Gamma(x) := \int_{0}^{\infty} u^{x-1} e^{-u} \, du$ for $\Re(x) > 0$ \cite[\S2]{Rainville1960}. The \emph{beta integral} is of 
 basic importance as a special function, in much the same way as the $\Gamma$-function, and may be defined so that 
\begin{equation}\label{betadefinition}
 B(p, q) := \int_{0}^{1} t^{p-1} (1-t)^{q-1} \, dt 
\end{equation}
 for $\Re(p) > 0$ and $\Re(q) > 0$ \cite[\S2.16]{Rainville1960}, so that 
\begin{equation}\label{betaevaluate}
 B(p, q) = \frac{ \Gamma(p) \Gamma(q) }{ \Gamma(p + q)}, 
\end{equation}
 again for $\Re(p) > 0$ and $\Re(q) > 0$. Our main family of dilogarithm identities is derived, as below, using \eqref{betadefinition} 
 together with the families of sextic integrals shown in Section \ref{sectionsextic} below. 

 The \emph{dilogarithm function} is a central object of study in our work and is given by the $m = 2$ case of the 
 \emph{polylogarithm function} such that $\operatorname{Li}_{m}(x) := \sum_{n=1}^{\infty} \frac{x^{n}}{n^m}$ for complex $x$ 
 such that $|x| \leq 1$. For a full review of background material on the dilogarithm, we again refer to the seminal monographs due to 
 Lewin concerning this function \cite{Lewin1958,Lewin1981,Lewin1991text}. By writing 
\begin{equation}\label{RogersLdefinition}
 L(x) := \frac{6}{\pi^2} \left( \operatorname{Li}_{2}(x) + \frac{1}{2} \ln x \ln(1-x) \right) 
\end{equation}
 to denote the Rogers $L$-function, a version of the five-term (Abel) functional equation such that $$ L(x) + L(y) = L(x y) + 
 L\left( \frac{x(1-y)}{1- x y} \right) + L\left( \frac{y(1-x)}{1-xy} \right) $$ is to be applied in our work. Other fundamental functional 
 equations for the dilogarithm that are to arise in our work include the duplication formula such that 
\begin{equation}\label{duplicationformula}
 2 \operatorname{Li}(x) + 2 \operatorname{Li}(-x) = \operatorname{Li}(x^{2}), 
\end{equation}
 together with Euler's reflection identity such that 
\begin{equation}\label{Eulerreflectionidentity}
 \operatorname{Li}_2(x) = -\operatorname{Li}_2(1 - x) + \frac{\pi^2}{6} - \ln(x) \ln(1 - x), 
\end{equation}
 along with Landen's dilogarithm identity such that 
\begin{equation}\label{Landenformula}
 \operatorname{Li}_{2}(-x) = -\operatorname{Li}_{2}\left( \frac{x}{1+x} \right) - \frac{1}{2} \ln^2(1+x). 
\end{equation}
 In addition to the Rogers $L$-function defined in \eqref{RogersLdefinition}, there are a number of well known special functions that 
 are closely related to $\operatorname{Li}_{2}$ and that are to be involved in our work. This includes the \emph{inverse tangent integral} 
 function $\text{Ti}_{2}$ \cite[\S2]{Lewin1981} such that $$ \text{Ti}(x) = \frac{1}{2i} \left( \operatorname{Li}_{2}(i 
 x) - \operatorname{Li}_{2}(-ix) \right) = \int_{0}^{x} \frac{ \operatorname{arctan} \, t}{t} \, dt$$ for $|x| \leq 1$, together with the 
 \emph{Legendre chi-function} $\chi_2$ \cite[\S1.8]{Lewin1981} such that $$ \chi_{2}(x) = \frac{1}{2} \operatorname{Li}_{2}(x) - 
 \frac{1}{2} \operatorname{Li}_{2}(-x) = \frac{1}{2} \int_{0}^{x} \ln\left( \frac{1+t}{1-t} \right) \frac{dt}{t} $$ for $|x| \leq 1$, along with 
 the \emph{Clausen function} $\text{Cl}_{2}$ \cite[\S4]{Lewin1981} such that $$ \text{Cl}_{2}(x) = \sum_{n=1}^{\infty} \frac{\sin n 
 x}{n^2} = -\int_{0}^{x} \ln\left( 2 \sin \frac{t}{2} \right) \, dt $$ for $|x| \leq 1$. 

 Mathematical objects referred to as \emph{dilogarithm ladders} are a main object to be explored in this paper, and our main 
 results include new ladder relations, inspired by the recent work of Campbell \cite{Campbell2025} on dilogarithm ladders. As in this 
 work by Campbell \cite{Campbell2025}, we reproduce the standard definition of a polylogarithm \emph{ladder} of a given weight $n$ 
 and index $N$ and $\mathbb{Q}$-algebraic argument $u$, and this refers to an expression of the form 
\begin{equation}\label{ladderdefinition}
 L_{n}(N, u) := \frac{\operatorname{Li}_{n}\left( u^{N} \right)}{N^{n - 1} } - \left( \sum_{r=0}^{N-1} \frac{A_{r} 
 \operatorname{Li}_{n}\left( u^{r} \right)}{r^{n-1}} + \frac{A_{0} \ln^{n}(u)}{n!} \right) 
\end{equation}
 for rational values $A_{r}$ such that the right-hand side of \eqref{ladderdefinition} reduces to 
\begin{equation}\label{Dscalar}
 L_{n} = \sum_{m=2}^{n} \frac{ D_{m} \ln^{n-m}(u) }{(n-m)!} 
\end{equation}
 for scalar coefficients $D_{m}$ \cite[p.\ 6]{Lewin1991text}. If such scalars are rational, then the ladder relation given by $ L_{n}(N, 
 u) $ reducing to \eqref{Dscalar} is said to be \emph{valid}. The valid ladder relations that we introduce as main results appear to be 
 inequivalent to previously known ladder relations, and we invite the interested reader to review past 
 and notable references related to 
 polylogarithm ladders 
 \cite{AbouzahraLewinXiao1987,CohenLewinZagier1992,Gangl2013,Lewin1993,Lewin1986,Lewin1991collection}. 

\section{Sextic integrals and dilogarithm identities}\label{sectionsextic}
 One of the main tools to be applied in this paper is given by the functional equations involved in Theorem \ref{maintheorem} below. 
 For our derivation of this Theorem, we require the integral identities such that 
\begin{equation}\label{integrandforsum1}
 \int_{0}^{h} \frac{(1 - y)^c y^d}{1 + g^2 (1 - y)^a y^b} \, dg = (1 - y)^{c - \frac{a}{2}} y^{d - \frac{b}{2}} 
 \operatorname{arctan}\left( h (1 - y)^{\frac{a}{2}} y^{\frac{b}{2}} \right) 
 \end{equation}
 for $\Re\left( \frac{ \left( 1 - y \right)^{-a/2} y^{-b/2} }{h} \right) \neq 0$ and such that 
\begin{equation}\label{integrandforsum2}
 \int_0^h \frac{(1 - y)^c y^d}{\sqrt{1 - g^2 (1 - y)^a y^b}} \, dg = (1 - y)^{c - \frac{a}{2}} y^{d - 
 \frac{b}{2}} \text{arcsin}\left( h (1 - y)^{\frac{a}{2}} y^{\frac{b}{2}} \right) 
 \end{equation}
 for $\Re\left( \frac{ \left( 1 - y \right)^{-a/2} y^{-b/2} }{h} \right) > 1$ and 
 for $\Re\left( \frac{1}{h \sqrt{ -(1-y)^{a} y^{b} }} \right) \neq 0$. 
 We refer to such integrals as \emph{sextic integrals}, in reference ot the sixfold combinations of parameters involved in each of 
 \eqref{integrandforsum1} and \eqref{integrandforsum2}. In view of the integrands among \eqref{integrandforsum1} and 
 \eqref{integrandforsum2}, we are to also make use of the power series identities such that 
\begin{equation}\label{powersum1}
 \sum_{n=0}^{\infty} (-1)^n g^{2n} y^{bn} (1 - y)^{an} = \frac{1}{1 + g^2 (1 - y)^a y^b}
\end{equation}
 and such that 
\begin{equation}\label{powersum2}
 \sum_{n=0}^{\infty} g^{2n} y^{bn} (1 - y)^{an} \frac{ \left( \frac{1}{2}
 \right)_n}{(1)_n} = \frac{1}{\sqrt{1 - g^2 (1 - y)^a y^b}} 
\end{equation} 
 for a suitable radius of convergence. Informally, by multiplying both sides of \eqref{powersum1} by the integrand factors indicated in 
 \eqref{displayw1} and \eqref{displayw2} below, and by then using the dominated convergence theorem to apply term-by-term 
 integration to the infinite sums that we obtain from \eqref{powersum1} and \eqref{powersum2}, 
 we would then use the beta integral identity in \eqref{betaevaluate}, 
 following a similar approach as in our past work \cite{Hakimoglupreprint}, 
 with the equality of the values $w_{1}$ and $w_{2}$ shown in 
 \eqref{displayw1} and \eqref{displayw2} providing a key tool in our work. 

 We make use of the notational shorthand such that 
\begin{equation*}
 \left[ \begin{matrix} \alpha, \beta, \ldots, \gamma \vspace{1mm} \\ 
 A, B, \ldots, C \end{matrix} \right]_{n} = \frac{ (\alpha)_{n} (\beta)_{n} 
 \cdots (\gamma)_{n} }{ (A)_{n} (B)_{n} \cdots (C)_{n}}. 
\end{equation*}
 From the sextic integral relation in \eqref{integrandforsum1}, we obtain that 
\begin{equation}\label{addedexplanation1}
 \frac{ \operatorname{arctan}\left( h (1 - y) \sqrt{y} \right) }{2 h y} = \int_{0}^{h} 
 \frac{1-y}{2 h \sqrt{y} \left(1 + g^2 (1-y)^2 y \right)} \, dg, 
\end{equation}
 and, from the power series expansion in \eqref{powersum1}, we obtain that 
\begin{equation}\label{addedexplanation2}
 \frac{1-y}{2 h \sqrt{y} \left(1 + g^2 (1-y)^2 y \right)} 
 = \frac{1}{2h } \sum_{n=0}^{\infty} (-1)^{n} g^{2n} y^{n - \frac{1}{2}} (1-y)^{2n+1}, 
\end{equation}
 and from \eqref{addedexplanation1} and \eqref{addedexplanation2} together, we obtain that 
\begin{equation}\label{displayw1}
 w_1 = \int_0^1 \frac{\arctan\left(h (1 - y) \sqrt{y}\right)}{ 2 h y} \, dy = 
 \int_{0}^{1} \frac{ (3x^2 - 1) \ln x }{1 + h^2 x^2 (1-x^2)^2} \, dx
 = \frac{ 2 s_1}{3}, 
\end{equation}
 where $s_1$ is such that 
\begin{equation}\label{displays1}
 s_1 = \sum_{n=0}^{\infty} 
 h^{2n} \left( -\frac{4}{27} \right)^{n} 
 \left[ \begin{matrix} \frac{1}{2}, \frac{1}{2}, 1 
 \vspace{1mm} \\ 
 \frac{5}{6}, \frac{7}{6}, \frac{3}{2} \end{matrix} \right]_{n}. 
\end{equation}
 A symmetric approach gives us the companion to \eqref{displayw1} such that 
\begin{equation}\label{displayw2}
 w_2 = \int_0^1 \frac{\arctan\left(h (1 - y) \sqrt{y}\right)}{h(1 - y)} \, dy 
 = \int_{0}^{1} \frac{ (1-3x^2) \ln(1-x^2) }{1 + h^2 x^2 (1-x^2)^2} \, dx = \frac{ 2 s_1}{3}, 
\end{equation}
 and the key insight associated with one of the main techniques in this paper is given by how $\frac{w_1}{w_2}$ is a constant, as this can 
 be used, as we demonstrate, to produce dilogarithm identities via a ``shifting'' applied related ${}_{4}F_{3}$-series. 

 Our method may be thought of as having the advantage of avoiding the double integrals required in the related work of Batir 
 \cite{Batir2005}. By iterating substitutions and taking multiple integrals that may be expressed as \eqref{displays1} in terms of 
 $g$, it is typically feasible to compute order-$n$ polylogarithms without having to compute a multivariable integral of 
 \eqref{powersum1}. For example, applying our iterative approach in a twofold manner using \eqref{integrandforsum2} and 
 \eqref{powersum2}, we obtain 
\begin{multline*}
 s_2 = \int_0^1 \frac{x(1 - 3x^2) \arcsin \left( g x (x^2 - 1) \right)}{x^2 - 1} \, dx = \\ 
 g\int_0^1 \frac{\left( \frac{3}{2}(x^2 - 1) + \ln(x^2 - 1) \right) (3x^2 - 1)}{\sqrt{1 - g^2 x^2 (x^2 - 1)^2}} \, dx, 
\end{multline*}
 writing 
\begin{equation*}
 s_2 = -\frac{4g}{15} \sum_{n=0}^{\infty} g^{2n} \left( \frac{4}{27} \right)^{n} 
 \left[ \begin{matrix} \frac{1}{2}, \frac{1}{2}, \frac{1}{2} 
 \vspace{1mm} \\ 
 \frac{7}{6}, \frac{11}{6}, \frac{3}{2} \end{matrix} \right]_{n}. 
\end{equation*}
 By setting $g = ig_0$, and by then enforcing the substitution $g_0 = \frac{1}{u(1-u^2)}$, the 
 denominator of the latter integrals 
 among \eqref{displayw1} and \eqref{displayw2} factors into three quadratics, yielding the product 
\begin{equation}\label{productforpartial}
 (u^2+ux+x^2-1)(u^2-ux+x^2-1)(u^2-x^2) 
\end{equation}
 This leads us to our first main result below. 

\begin{theorem}\label{maintheorem}
 Let 
\begin{align*}
 A & := \operatorname{Li}_2\left(\frac{\sqrt{4 - 3u^2} + u^2 - 2}{u (u + 1)}\right) 
 - \operatorname{Li}_2\left(\frac{\sqrt{4 - 3u^2} + u^2 - 2}{(u - 1) u}\right), \\ 
 B & := 
 \frac{1}{2} \ln^2\left(\frac{-i \sqrt{3 u^2-4}+u+2}{4} \right) - 
 \frac{1}{2} \ln ^2\left(\frac{- i \sqrt{3 u^2 - 4} - u^2 +2}{u 
 (u+1)}\right) + \\ 
 & \quad \frac{1}{2} \ln ^2\left(\frac{i \sqrt{3 u^2 - 4} + u + 2}{4} \right)-\frac{\pi ^2}{6}, \\ 
 J & := \frac{1}{4} \ln^2 \left( \frac{-4}{-i \sqrt{3 u^2 - 4} + u + 2} \right) + 
 \frac{1}{4} \ln^2 \left( \frac{-4}{i \sqrt{3 u^2 - 4} + u + 2} \right) - \\ 
 &\quad \frac{1}{4} \ln^2 \left( \frac{-4}{-i \sqrt{3 u^2 - 4} - u + 2} \right) - \frac{1}{4} 
 \ln^2 \left( \frac{-4}{i \sqrt{3 u^2 - 4} - u + 2} \right), \\ 
 C & := - \frac{1}{2} \ln \left( \frac{ - i \sqrt{3 u^2 - 4} + u + 
 2}{4} \right) \ln \left( \frac{ i \sqrt{3 u^2 - 4} - u + 2}{4} \right) - \\ 
 & \quad \frac{1}{2} \ln \left( \frac{i \sqrt{3 u^2 - 4} + u + 2}{4} \right) 
 \ln \left( \frac{- i \sqrt{3 u^2 - 4} - u + 2}{4} \right) + \frac{\pi^2}{6}, \\ 
 H & := 
 \frac{1}{4} \ln^2\left(\frac{2}{u-1}\right)-\frac{1}{4} \ln ^2\left(-\frac{2}{u+1}\right)+\frac{1}{2} \ln \left(\frac{-u + 1}{2} \right) 
 \ln \left(\frac{u+1}{2}\right) - \\ 
 & \quad \frac{\pi i}{2} \ln \left(\frac{u + 
 1}{u-1}\right)-\ln (2) \ln \left(\frac{u+1}{u-1}\right)-\frac{\pi^2}{12}, \\ 
 D & := \operatorname{Li}_2 \left( \frac{1 - u}{2} \right) 
 + \ln(2) \ln \left( \frac{u - 1}{u + 1} \right) + \left( 2 \ln(2) + \frac{\pi i}{2} \right) \ln \left( \frac{u + 
 1}{u - 1} \right), \\ 
 K & := \frac{1}{2} \operatorname{Li}_2 \left( \frac{2}{- \sqrt{4 - 3 u^2} + u} \right) - 
 \frac{1}{2} \operatorname{Li}_2 \left( \frac{2}{\sqrt{4 - 3 u^2} - u} \right) \\ 
 &\quad - \frac{1}{2} \operatorname{Li}_2 \left( \frac{-2}{\sqrt{4 - 3 u^2} + u} \right) + 
 \frac{1}{2} \operatorname{Li}_2 \left( \frac{2}{\sqrt{4 - 3 u^2} + u} \right) + 
 \frac{1}{2} \operatorname{Li}_2 \left( \frac{-1}{u} \right) - 
 \frac{1}{2} \operatorname{Li}_2 \left( \frac{1}{u} \right). 
\end{align*} 
 Then, for all $u \in \mathbb{C} \setminus \mathbb{R}$ such that $\Re u \geq 0$ and such that the power series 
 in \eqref{displays3} converges, we have that 
\begin{equation}\label{mainABCD}
 A + B + J + C + H + D = s_3 = K, 
\end{equation}
 and that 
\begin{equation}\label{MequalsAB}
 M = A + B = 
 - \operatorname{Li}_{2}\left( \frac{\sqrt{4 - 3 u^2} - u + 2}{4} \right) 
 - \operatorname{Li}_{2}\left( \frac{-\sqrt{4 - 3 u^2} - u + 2}{4} \right), 
\end{equation}
 writing 
\begin{equation}\label{displays3}
 s_3 = \frac{2}{3u(u^2-1)} \sum_{n=0}^{\infty} \left( \frac{1}{u(u^2-1)} \right)^{2n} \left( \frac{4}{27} \right)^{n} 
 \left[ \begin{matrix} \frac{1}{2}, \frac{1}{2}, 1 
 \vspace{1mm} \\ 
 \frac{5}{6}, \frac{7}{6}, \frac{3}{2} \end{matrix} \right]_{n}. 
\end{equation}
 Moreover, for $u \in \mathbb{R}_{\geq 0}$ such that \eqref{displays3} converges, by letting $v \in \mathbb{C} \setminus \mathbb{R}$ 
 approach $u$ and taking the corresponding limits on all sides of the equalities among \eqref{mainABCD}, these limits are equal, 
 and similarly for \eqref{MequalsAB}. 
\end{theorem}

\begin{proof}
 This follows from our sextic integral identities, by applying partial fraction decomposition to the reciprocal of the product in 
 \eqref{productforpartial}, and by then applying this partial fraction decomposition within the appropriate integrands in 
 \eqref{displayw1} and \eqref{displayw2}, and by then evaluating the indefinite integrals corresponding to the definite integrals we 
 obtain by expanding the aforementioned integrands. More explicitly, in the latter equality in 
 \eqref{displayw1}, we set $-\frac{i}{u \left(u^2-1\right)}$, yielding 
\begin{multline*}
 \int_{0}^{1} \frac{\left(3 x^2 - 
 1\right) \ln (x)}{\left(u^2+u x+x^2-1\right) \left(u^2-u x+x^2-1\right) \left(u^2-x^2\right)} \, dx = \\ 
 \frac{2}{3 (u-1)^2 u^2 (u+1)^2} \sum_{n=0}^{\infty} 
 \left( \frac{1}{u(u^2-1)} \right)^{2n} \left( \frac{4}{27} \right)^{n} 
 \left[ \begin{matrix} \frac{1}{2}, \frac{1}{2}, 1 
 \vspace{1mm} \\ 
 \frac{5}{6}, \frac{7}{6}, \frac{3}{2} \end{matrix} \right]_{n}, 
\end{multline*}
 and, by applying partial fraction decomposition 
 to the above integrand, an indefinite integral 
 corresponding to the resultant integrand is equivalent to 
\begin{multline*}
 \frac{1}{2 u-2 u^3} \Bigg( \operatorname{Li}_2\left(-\frac{2 x}{u-\sqrt{4-3 u^2}}\right) - 
 \operatorname{Li}_2\left(\frac{2 x}{u-\sqrt{4-3 u^2}}\right)+\operatorname{Li}_2\left(-\frac{2
 x}{u+\sqrt{4-3 u^2}}\right) - \\ 
 \operatorname{Li}_2\left(\frac{2 x}{u+\sqrt{4-3
 u^2}}\right)-\operatorname{Li}_2\left(-\frac{x}{u}\right)+\operatorname{Li}_2\left(\frac{x}{u}\right) + \ln (x) \ln \left(\frac{\sqrt{4-3 u^2}-u-2
 x}{\sqrt{4-3 u^2}-u}\right) - \\ 
 \ln (x) \ln \left(\frac{\sqrt{4-3 u^2}+u-2 x}{\sqrt{4-3 u^2}+u}\right)-\ln (x) \ln \left(\frac{\sqrt{4-3
 u^2}-u+2 x}{\sqrt{4-3 u^2}-u}\right) + \\ 
 \ln (x) \ln \left(\frac{\sqrt{4-3 u^2}+u+2 x}{\sqrt{4-3 u^2}+u}\right)+\ln (x) \ln
 \left(\frac{u-x}{u}\right)-\ln (x) \ln \left(\frac{u+x}{u}\right) \Bigg). 
\end{multline*}
 By then setting $x \to 1$ and $x \to 0$, we obtain an equivalent version of the desired equality $s_3 = K$, 
 and similarly for 
 $ A + B + J + C + H + D = s_3$ and for $ M = A + B$, for the case whereby
 $u \in \mathbb{C} \setminus \mathbb{R}$ satisfies the specified convergence condition, 
 and an analytic continuation then gives us the desired result for real arguments. 
\end{proof}

\section{Applications}
 Using Theorem \ref{maintheorem}, we derive evaluations for accelerated ${}_{4}F_{3}$-series and dilogarithm identities, as below. 

\subsection{A full solution to a problem due to Sun}
 As it turns out, the $u = \frac{1}{\sqrt{3}}$ case Theorem \ref{maintheorem} can be used to formulate a full solution to a problem due 
 to Sun \cite{Sun2021,Sun2015} that was partially solved by Campbell \cite{Campbell2023}, as later explained. 

\begin{theorem}\label{theoremSunCampbell}
 The closed-form evaluation
 $$ 8 \operatorname{Li}_2\left(2-\sqrt{3}\right)-\operatorname{Li}_2\left(4 \sqrt{3} - 7 \right) 
 = \frac{5 \pi ^2}{12}+\ln \left(2-\sqrt{3}\right) \ln \left(2+\sqrt{3}\right) $$
 holds true.
\end{theorem}

\begin{proof}
 For the $u = \frac{1}{\sqrt{3}}$ case Theorem \ref{maintheorem}, 
 the left-hand side of \eqref{mainABCD} 
 may be expressed in terms of 
 $ -2 \operatorname{Li}_2 \left( 7 - 4 \sqrt{3} \right) + 
 4 \operatorname{Li}_2 \left( \sqrt{3} - 2 \right) $ and a closed-form expression, 
 and this follows from the known relation 
\begin{equation*}
-2 \operatorname{Li}_2 \left( \frac{3 - \sqrt{3}}{6} \right) = 2 \operatorname{Li}_2 \left( \sqrt{3} - 2 \right) + 
 \ln^{2} \left( \frac{3 + \sqrt{3}}{6} \right). 
\end{equation*}
 The right-hand side of \eqref{mainABCD}, for the same input value 
 $u = \frac{1}{\sqrt{3}}$ may be rewritten as 
 $$ K = \operatorname{Li}_2\left(-\sqrt{3}\right)-\frac{1}{2}
 \operatorname{Li}_2\left(-\frac{\sqrt{3}}{2}\right) + 
 \frac{\operatorname{Li}_2\left(\frac{\sqrt{3}}{2}\right)}{2}-\operatorname{Li}_2\left(\sqrt{3}\right). $$
 Using the well known $4$- term functional equation 
\begin{equation}\label{wellknown4term}
 \operatorname{Li}_2(z) - \operatorname{Li}_2(-z) + 
 \operatorname{Li}_2\left(\frac{1 - z}{1 + z}\right) - 
 \operatorname{Li}_2\left(\frac{-(1 - z)}{1 + z}\right) = \frac{\pi^2}{4} + \ln(z) \ln\left(\frac{1 + z}{1 - z}\right), 
\end{equation}
 we obtain, by setting $z = \frac{\sqrt{3}}{2}$ and $z=\sqrt{3}$, that the right-hand side of \eqref{mainABCD} may be rewritten as 
\begin{equation*}
 \operatorname{Li}_2\left(7 - 4 \sqrt{3}\right) - \operatorname{Li}_2\left(4 \sqrt{3} - 7 \right) - 
 2 \, \operatorname{Li}_2\left(\sqrt{3} - 2\right) + 2 \, \operatorname{Li}_2\left(2 - \sqrt{3}\right). 
\end{equation*}
 Equating our equivalent expressions for the right-hand and left-hand sides of \eqref{mainABCD}, we obtain 
\begin{equation*}
 3 \, \operatorname{Li}_2\left(7 - 4 \sqrt{3}\right) - 
 \operatorname{Li}_2\left(4 \sqrt{3} - 7\right) - 
 6 \, \operatorname{Li}_2\left(\sqrt{3} - 2\right) + 2 \, \operatorname{Li}_2\left(2 - \sqrt{3}\right)
\end{equation*}
 and 
\begin{equation*}
 3 \, \operatorname{Li}_2\left(7 - 4 \sqrt{3}\right) = 6 \, 
 \operatorname{Li}_2\left(2 - \sqrt{3}\right) + 6 \, \operatorname{Li}_2\left(\sqrt{3} - 2\right), 
\end{equation*}
 and this gives us an equivalent version of the desired result. 
\end{proof}

 It was shown by Campbell \cite{Campbell2023} that the open problem due to Sun \cite{Sun2021,Sun2015} given by proving the 
 conjectured formula 
\begin{equation}\label{Sunconjectureharmonic}
 \sum_{n=0}^{\infty} \frac{ \binom{2n}{n} }{ (2n+1) 8^{n} } \left( \sum_{0 \leq k < n} 
 \frac{(-1)^k}{2k+1} - \frac{(-1)^n}{2n+1} \right) = -\frac{\sqrt{2}}{16} \pi^2 
\end{equation}
 is equivalent to the problem of proving the purported closed-form evaluation 
\begin{equation}\label{discoveredwithAu}
 \operatorname{Li}_{2}\left( \frac{\sqrt{3} + 2}{4} \right) - 8 \operatorname{Li}_{2}(2 - \sqrt{3}) 
 = -\frac{\pi^2}{4} - 2 \ln^{2}(2) + \frac{5}{2} \ln^2(\sqrt{3} + 2) - 2 \ln(\sqrt{3} + 2) \ln 2 
\end{equation}
 discovered experimentally using a Mathematica package given by Au, and referring to Campbell's work \cite{Campbell2023} for details. 
 It can be shown that Theorem \ref{theoremSunCampbell} is equivalent to \eqref{discoveredwithAu}, thus providing a full proof of Sun's 
 conjectured formula in \eqref{Sunconjectureharmonic}. 

\subsection{On an identity attributed to Lima}
 The $u = 
 \sqrt{2}$ case of Theorem \ref{maintheorem}, we obtain, from the $A$-value involved in Theorem \ref{maintheorem} together 
 with the duplication formula in \eqref{duplicationformula}, that 
\begin{multline*}
 \operatorname{Li}_2\left(-i\left(1 + \sqrt{2}\right)\right) - \operatorname{Li}_2\left(-i\left(-1 + \sqrt{2}\right)\right) = \\ 
 \frac{1}{2} \operatorname{Li}_2\left(-3 - 2 \sqrt{2}\right) + \frac{\pi^2}{6} + \frac{1}{2} \ln^2\left(i\left(\sqrt{2} - 1\right)\right) 
\end{multline*}
 and that 
\begin{equation*}
\operatorname{Li}_2\left(\frac{1 - \sqrt{2}}{2} 
 \right) = -\operatorname{Li}_2\left(3 - 2 \sqrt{2}\right) - \frac{1}{2} \ln^2\left(\frac{1}{2} + \frac{1}{\sqrt{2}}\right)
\end{equation*}
 For the right-hand side of the identity in \eqref{mainABCD}, the associated integral admits an evaluation that can be derived using a 
 factorization property associated with roots lying on the unit disk, yielding 
\begin{equation*}
 \int_0^1 \frac{(1 - 3x^2) \ln(x)}{(x^2 - 2)(x^4 + 1)} \, dx = 
 \frac{\sqrt{2}}{4} \left( -2 \, \operatorname{Li}_2\left( \frac{1}{\sqrt{2}} \right) + \frac{7\pi^2}{24} - \frac{(\ln(2))^2}{4} \right)
\end{equation*}
 We thus obtain the formula 
\begin{equation}\label{attributedLima}
 \operatorname{Li}_2\left( \frac{1}{\sqrt{2}} \right) - \operatorname{Li}_2\left( 1 - \sqrt{2} \right) - 
 \frac{\pi^2}{6} + \frac{\ln^2(2)}{8} - \frac{1}{2} \ln(2) \ln\left( \sqrt{2} - 1 \right) = 0 
\end{equation}
 that has been attributed to Lima \cite{Lima2012} and subsequently explored by Stewart \cite{Stewart2022} and generalized by Lima 
 \cite{Lima2024} and that can be derived using a classical relation for the Legendre's $\chi$-function together with the Landen 
 transform in \eqref{Landenformula}. 

\subsection{Zagier's four-term identity}
 Setting $u = \frac{4}{3}$ in 
 Theorem \ref{maintheorem}, it is possible, as below, to derive Zagier's four-term identity such that 
 \begin{equation}\label{Zagierfour}
 3 L\left(-\frac{1}{6}\right) - L\left(\frac{1}{8}\right) + L \left(\frac{1}{9}\right) + L \left(\frac{1}{28}\right) = -\frac{\pi^2}{12}. 
\end{equation}
 Using Theorem \ref{maintheorem}, we 
 provide a sketch of a proof of the Zagier identity in \eqref{Zagierfour}. On the left-hand side, 
 we apply the relation \eqref{MequalsAB} from Theorem \ref{maintheorem}, and we set $u = 
 \frac{4}{3}$, $s = 2$, $p = 3$, and $z = 
 \frac{1}{3}$ in the formula 
\begin{equation}\label{sigmasumLi}
\sum_{m=0}^{p-1} \operatorname{Li}_s(z e^{2\pi i m/p}) = p^{1-s} \operatorname{Li}_s(z^p),
\end{equation}
 Plugging in $u = \frac{4}{3}$ and using the four-term relation in \eqref{wellknown4term}, we obtain 
 \begin{multline*}
 - \mathrm{Li}_2\left(-\frac{1}{7}\right) + \mathrm{Li}_2\left(\frac{1}{7}\right) 
 - \mathrm{Li}_2\left(\frac{-6 - 3i\sqrt{3}}{7}\right) 
 + \mathrm{Li}_2\left(\frac{6 - 3i\sqrt{3}}{7}\right) \\ 
 - \mathrm{Li}_2\left(\frac{-6 + 3i\sqrt{3}}{7}\right) 
 + \mathrm{Li}_2\left(\frac{6 + 3i\sqrt{3}}{7}\right) 
 - \frac{\pi^2}{4} + \ln\left(\frac{4}{3}\right)\ln(7). 
\end{multline*}
 From Abel's formula, we find that 
\begin{multline*}
 \mathrm{Li}_2\left(\frac{z}{1 - z} \cdot \frac{w}{1 - w}\right) = \\
 \mathrm{Li}_2\left(\frac{z}{1 - w}\right) 
 + \mathrm{Li}_2\left(\frac{w}{1 - z}\right) 
 - \mathrm{Li}_2(z) - \mathrm{Li}_2(w) 
 - \ln(1 - z)\ln(1 - w), 
\end{multline*}
 writing 
\begin{equation*}
w = \frac{1}{\frac{1}{7}\left(-6 - 3i\sqrt{3}\right)} \ \ \ \text{and} \ \ \ 
z = \frac{1}{\frac{1}{7}\left(-6 + 3i\sqrt{3}\right)}, 
 \end{equation*}
 and we omit the elementary transformations we have applied, for the sake of brevity. This yields an expression that is again amenable 
 to \eqref{sigmasumLi}, so we are able to eliminate the the complex conjugate terms: 
\begin{multline*}
 \frac{\operatorname{Li}_2\left(\frac{1}{8}\right)}{3} - \operatorname{Li}_2\left(\frac{1}{4}\right) - \operatorname{Li}_2\left(\frac{7}{-6 - 3i\sqrt{3}}\right) 
 - \operatorname{Li}_2\left(\frac{7}{-6 + 3i\sqrt{3}}\right) - \frac{\pi^2}{12} + \frac{\ln^2(2)}{2}\\ 
 - \left( -\ln\left(\frac{3}{2}\right) + \frac{\ln(7)}{2} - i \cot^{-1}\left(\frac{5}{\sqrt{3}}\right) \right) \left( -\ln\left(\frac{3}{2}\right) + 
 \frac{\ln(7)}{2} + i \cot^{-1}\left(\frac{5}{\sqrt{3}}\right) \right)=0. 
\end{multline*}
 Mimicking the above derivation with the use of Roger's formula, 
 we obtain that 
\begin{multline*}
 \operatorname{Li}_2(z w) - \operatorname{Li}_2(w) - \operatorname{Li}_2(z) + \operatorname{Li}_2\left(\frac{w(1 - z)}{1 - w z}\right) + \operatorname{Li}_2\left(\frac{z(1 - 
 w)}{1 - w z}\right) = \\ 
 -\ln\left(\frac{1 - z}{1 - w z}\right) \ln\left(\frac{1 - w}{1 - w z}\right), 
\end{multline*}
 writing 
\begin{equation*}
 w = 1 - \frac{3}{7}(2 - i\sqrt{3}) \ \ \ \text{and} \ \ \ \quad z = 1 - \frac{3}{7}(2 + i\sqrt{3}). 
\end{equation*}
 We thus have that 
\begin{equation*}
\begin{aligned}
\frac{\pi^2}{4} - \frac{\ln^2(2)}{2} + \frac{(i \pi + \ln(2))^2}{2} 
+ \left( -i \cot^{-1}\left(\frac{2}{\sqrt{3}}\right) + \frac{\ln(7)}{2} \right)\left( i \cot^{-1}\left(\frac{2}{\sqrt{3}}\right) + \frac{\ln(7)}{2} \right)\\
+ \operatorname{Li}_2\left(\frac{4}{7}\right) + \frac{\operatorname{Li}_2(8)}{3} - \operatorname{Li}_2\left(\frac{1}{7} - \frac{3i}{7}\sqrt{3}\right) - \operatorname{Li}_2\left(\frac{1}{7} + \frac{3i}{7}\sqrt{3}\right)=0. 
\end{aligned}
\end{equation*}
 We then apply the known identities 
\begin{align}
 \operatorname{Li}_2\left(\frac{4}{3}\right) + \operatorname{Li}_2(4) & = \frac{1}{2}\left( \pi^2 
 - \ln^2(3) \right), \nonumber \\ 
 -\operatorname{Li}_2\left(\frac{4}{7}\right) & = \operatorname{Li}_2\left(-\frac{4}{3}\right) + 
 \frac{1}{2} \ln^2\left(\frac{3}{7}\right), \nonumber \\ 
 -\frac{{\pi}^2}{18}+\frac{(\ln 3)^2}{6}. & = \operatorname{Li}_2\left(-\frac{1}{3}\right) - 
 \frac{1}{3}\operatorname{Li}_2\left(\frac{1}{9}\right), \text{and} \nonumber \\ 
 \frac{{\pi}^2}{18}+2\ln2\cdot\ln3-2(\ln 2)^2-\frac{2}{3}(\ln 3)^2 
 & = \operatorname{Li}_2\left(\frac{1}{4}\right) + 
 \frac{1}{3}\operatorname{Li}_2\left(\frac{1}{9}\right), \label{previously4p28}
\end{align}
 together with $$ -\operatorname{Li}_2\left(-\frac{4}{3}\right) = \operatorname{Li}_2\left(\frac{1}{7}\right) - \operatorname{Li}_2\left(-\frac{1}{7}\right) - 
 \operatorname{Li}_2\left(\frac{4}{3}\right) - \ln(3)\ln(7) + \ln(4)\ln(7) + \frac{\pi^2}{4}. $$ Converting the arguments $-\frac{1}{3}$ and 
 $\frac{1}{9}$ to obtain the arguments $\frac{1}{4}$ and $\frac{1}{9}$ may be achieved with a classical functional equation for the 
 dilogarithm, and similarly for the arguments $-\frac{1}{7}$ and $\frac{1}{8}$. 

 The associated ${}_{4}F_{3}$-series is as follows: 
\begin{equation*}
\begin{aligned}
 \frac{28 \, }{27} \operatorname{Li}_2\left(\frac{1}{8}\right) - \frac{7 \pi^2}{81} + 
 \frac{14 \ln^2(2)}{3} - \frac{7 \ln^2(7)}{18} + \frac{56}{27} \pi \cot^{-1}\left(\frac{5}{\sqrt{3}}\right) - 
 \frac{14}{9} \left(\cot^{-1}\left(\frac{5}{\sqrt{3}}\right)\right)^2 \\ 
 = 2\sum_{n=0}^{\infty} \left( \frac{27}{196} \right)^{n} 
 \left[ \begin{matrix} \frac{1}{2}, \frac{1}{2}, 1 
 \vspace{1mm} \\ 
 \frac{5}{6}, \frac{7}{6}, \frac{3}{2} \end{matrix} \right]_{n}. 
\end{aligned}
\end{equation*}
 To prove this without the Zaiger identity, we have that 
 the left-hand side may be expressed with 
\begin{equation*}
 \operatorname{Li}_2\left(-\frac{1}{3}\right) + \frac{1}{3} \, \operatorname{Li}_2\left(\frac{1}{28}\right) + \operatorname{Li}_2\left(-\frac{1}{6}\right), 
\end{equation*}
 disregarding logarithmic and other non-dilogarithmic terms for the sake of brevity. Then, using the above formulation of Abel's 
 formula, we make use of 
\begin{equation*}
 \operatorname{Li}_2\left(\frac{1}{28}\right) = \operatorname{Li}_2\left(\frac{1}{7}\right) + 
 \operatorname{Li}_2\left(\frac{1}{4}\right) - \operatorname{Li}_2\left(\frac{1}{9}\right) - 
 \operatorname{Li}_2\left(\frac{2}{9}\right) - \ln\left(\frac{8}{9}\right)\ln\left(\frac{7}{9}\right) 
\end{equation*}
 and 
\begin{equation*}
 \operatorname{Li}_2\left(\frac{2}{9}\right) = \operatorname{Li}_2\left(\frac{1}{3}\right) + 
 \operatorname{Li}_2\left(\frac{2}{3}\right) - \operatorname{Li}_2\left(\frac{4}{7}\right) - 
 \operatorname{Li}_2\left(\frac{1}{7}\right) - \ln\left(\frac{3}{7}\right)\ln\left(\frac{6}{7}\right). 
\end{equation*}
 Informally, since the bases $-\frac{1}{6}$ and $-\frac{1}{3}$ 
 cancel with one another, we are left with a basis given by $\frac{1}{8}$. 

 Using \eqref{previously4p28} and the well-known identity $3 \operatorname{Li}_2\left(\frac{1}{2}\right) - 
 3 \operatorname{Li}_2\left(\frac{1}{4}\right) - \operatorname{Li}_2\left(\frac{1}{8}\right) + 
 \frac{1}{2} \operatorname{Li}_2\left(\frac{1}{64}\right) - \frac{\pi^2}{12} = 0$, it is possible to compress Zagier's four-term identity into a 
 three-term identity. In this direction, we obtain that 
\begin{equation*}
\frac{9}{2} \ln^2(2) - 6 \ln(2) \ln(3) + 2 \ln^2(3) + 
 \frac{1}{2} \operatorname{Li}_2\left(\frac{1}{64}\right) = 
 -\operatorname{Li}_2\left(\frac{1}{9}\right) + \operatorname{Li}_2\left(\frac{1}{8}\right)
\end{equation*}
 Although it is always possible to derive a 
 three-term identity with rational bases using Roger's identity, (e.g. by letting $z=1-m$ and $w=m$), there is no obvious way to prove 
 Zagier's identity with trivial manipulations of Roger's formula and elementary transformations. 

\subsection{Identities related of an evaluation for the inverse tangent integral}
 Setting $u= i\sqrt{3}$ it is possible to derive a relation between $\mathbb{Q}(\sqrt{13})$ 
 and $\mathbb{Q}(\sqrt{3})$ 
 that gives a fast series for $\rm{Cl}_2\big(\tfrac{\pi}3\big)$ that bypasses the summation of the multiple angle formula. 

 For the right-hand side, we obtain 
\begin{equation}\label{equivalentTi}
 \Im \left(\operatorname{Li}_2\left(\frac{i}{\sqrt{3}}\right) 
 - \operatorname{Li}_2\left(\frac{-i}{\sqrt{3}}\right)\right)= \frac{5}{3} \, \Im\operatorname{Li}_2\left(\frac{1 + i\sqrt{3}}{2}\right) - \frac{\pi \ln(3)}{6}, 
\end{equation}
 noting that the left-hand side of \eqref{equivalentTi} may be expressed in an equivalent way with 
 $\Im \text{Ti}_{2}\left( \frac{1}{\sqrt{3}} \right)$, and that the formula in \eqref{equivalentTi} may be proved in an equivalent way using 
 Ramanujan's identity \cite{Ramanujan1915} such that $$ \sum_{n=0}^{\infty} \frac{ \left( \frac{1}{2} \right)_{n} }{n!} \frac{ \cos^{2n+1} x 
 + \sin^{2n+1} x }{(2n+1)^2} = \text{Ti}_{2}(\tan x) + \frac{1}{2} \pi \ln(2 \cos x), $$ as noted in an equivalent way by Campbell 
 \cite{Campbell2022}. From \eqref{equivalentTi} we obtain that 
\begin{equation*}
\begin{aligned}
 \operatorname{Li}_2\left(-\frac{1}{12} \left(3 + i\sqrt{3}\right)\left(5 + \sqrt{13}\right)\right) - \operatorname{Li}_2\left(\frac{1}{12}\left(3 + 
 i\sqrt{3}\right)\left(5 + \sqrt{13}\right)\right) \\ 
 = \operatorname{Li}_2\left(-\frac{2i + \sqrt{3} + i\sqrt{13}}{-2i + \sqrt{3} + i\sqrt{13}}\right) - \operatorname{Li}_2\left(\frac{2i + \sqrt{3} + 
 i\sqrt{13}}{-2i + \sqrt{3} + i\sqrt{13}}\right). 
\end{aligned}
\end{equation*}
 For the left-hand side, the two-term dilogarithm combination given by the $A$-expression in Theorem 
 \ref{maintheorem} is used. 
 Consequently, we obtain the equality of 
\begin{multline*} 
 2\Im \, \operatorname{Li}_2\left(\frac{5 - \sqrt{13}}{\sqrt{12}} e^{\frac{\pi i}{6}}\right) - \Im\operatorname{Li}_2\left(\frac{1 + i\sqrt{3}}{2}\right) + \\ 
 \frac{\pi}{12} \ln\left(\frac{4}{3}\right) + \text{arccosh}\left(\frac{19}{6}\right) \operatorname{arctan}\left(\frac{\sqrt{13} - 2}{\sqrt{27}}\right) 
\end{multline*} 
 and 
\begin{multline*} 
 \Im \text {Li}_2\left(\frac{ \sqrt{13}-5}{\sqrt{12}} e^{\frac{\pi i}{6}}\right) - 
 \Im \operatorname{Li}_2\left(\frac{5 - \sqrt{13}}{\sqrt{12}} e^{\frac{\pi i}{6}} 
 \right) + \frac{5}{6}\Im\operatorname{Li}_2\left(\frac{1 + i\sqrt{3}}{2}\right) \\ 
 + \, \frac{\pi}{6} \ln(2) - \frac{1}{2} \operatorname{arctan}\left(\sqrt{\frac{3}{13}}\right) \ln\left(\frac{1}{6}(19 + 5\sqrt{13})\right) 
 - \frac{\pi}{12} \ln(3), 
\end{multline*}
 and, in turn, the above expression is equal to $$ \frac{1}{6\sqrt{3}}\sum_{n=0}^{\infty} \left( \frac{-1}{324} \right)^{n} \left[ 
 \begin{matrix} \frac{1}{2}, \frac{1}{2}, 1 \vspace{1mm} \\ \frac{5}{6}, \frac{7}{6}, \frac{3}{2} \end{matrix} \right]_{n} =4 
 \sqrt{3} \int_{0}^{1} \frac{(3x^2 - 1)\ln(x)}{48 + x^2(1 - x^2)^2} \, dx. $$ We can then use the duplication formula for $\operatorname{Li}_{2}$, 
 and this has the effect of a series acceleration, because, by taking the imaginary part of $e^{\frac{n\pi i}{3}}$ for positive integers $n$, 
 every third term is ignored. This gives us that 
\begin{multline*}
 \pi \ln\left(\frac{4}{3}\right) + 3\ln\left(\frac{19 + 5\sqrt{13}}{6} \right)\left( \operatorname{arctan}\left(\frac{3\sqrt{3}}{16 + 
 5\sqrt{13}}\right)\right) + \\ 
 3\cdot\Im\left[\operatorname{Li}_2\left(\frac{1}{12}(5 - \sqrt{13})^2 e^{\frac{\pi i}{3}}\right)\right] - \rm{Cl}_2\big(\frac{\pi}3\big) = 
 \frac{2}{\sqrt{3}}\sum_{n=0}^{\infty} \left( \frac{-1}{324} \right)^{n} \left[ \begin{matrix} \frac{1}{2}, \frac{1}{2}, 1 \vspace{1mm} \\ 
 \frac{5}{6}, \frac{7}{6}, \frac{3}{2} \end{matrix} \right]_{n} 
\end{multline*}
 and that 
\begin{equation*}
\begin{aligned}
\Im\left[\operatorname{Li}_2\left(\frac{1}{12}(5 - \sqrt{13})^2 e^{\frac{\pi i}{3}}\right)\right] = 
 \sqrt{3}\sum_{n=1}^{\infty} \frac{(19 - 5 \sqrt{13})^{3n - 2} }{(-216)^n} \left( \frac{5 \sqrt{13} - 
 19}{(1 - 3n)^2} - \frac{6}{(2 - 3n)^2} \right)
\end{aligned}
\end{equation*}

\subsection{A series of convergence rate $\frac{1}{243}$}
 The $u = 
 2$ case of Theorem \ref{maintheorem} yields the ${}_{4}F_{3}$-evaluation such that 
\begin{multline*}
 \sum_{n=0}^{\infty} \left( \frac{1}{243} \right)^{n} 
 \left[ \begin{matrix} \frac{1}{2}, \frac{1}{2}, 1 
 \vspace{1mm} \\ 
 \frac{5}{6}, \frac{7}{6}, \frac{3}{2} \end{matrix} \right]_{n} 
 = \frac{9 \operatorname{Li}_2\left(\frac{1}{4}\right)}{4}-\frac{3 \pi ^2}{4}+\frac{9 \ln^2(2)}{2}-\frac{9 \ln ^2(3)}{8} + \\ 
 \frac{9}{2} \operatorname{arctan}^{2}\left(\frac{1}{\sqrt{2}}\right) + 
 9 \operatorname{arctan}\left(\frac{1}{\sqrt{2}}\right) \operatorname{arctan}\left(\sqrt{2}\right). 
\end{multline*}
 This can be shown to follow from \eqref{MequalsAB}, as the right-hand side of the required dilogarithm identity may be verified 
 using Roger's five-term identity. 

\section{Elliptic ${}_{4}F_{3}$-series}
 In view of how we have applied ${}_{4}F_{3}$-functions of the form 
\begin{equation*}
 {}_{4}F_{3}\!\!\left[ \begin{matrix} 
 1, 1, \frac{1}{2}, \frac{1}{2} \vspace{1mm}\\ 
 b_1, b_2, b_3 \end{matrix} \ \Bigg| \ z \right], 
\end{equation*}
 this leads us to consider similar applications of ${}_{4}F_{3}$-functions of the form 
\begin{equation}\label{alternativehypergeometric}
 {}_{4}F_{3}\!\!\left[ \begin{matrix} 
 1, \frac{1}{2}, \frac{1}{2}, \frac{1}{2} \vspace{1mm}\\ 
 b_1, b_2, b_3 \end{matrix} \ \Bigg| \ z \right]. 
\end{equation}
 However, the partial fraction decomposition-based methods applied above are not applicable with the use of hypergeometric 
 functions of the form indicated in \eqref{alternativehypergeometric}. Informally, this is because a required integral may be 
 elliptic for almost all choices of $z$. Using \eqref{integrandforsum2}, together with 
\begin{equation*}
 \int_{0}^{1} \frac{\arcsin (g \, x(1 - x^2))}{{g x}} dx = \int_{0}^{1} \frac{(-1 + 3x^2)\ln(x)}{\sqrt{1 - g^2 x^2(1 - x^2)^2}} \, dx = 
 \frac{2}{3}\sum_{n=0}^{\infty} \left( \frac{4g^2 }{27} \right)^{n} \left[ \begin{matrix} \frac{1}{2}, \frac{1}{2}, \frac{1}{2} \vspace{1mm} \\ 
 \frac{5}{6}, \frac{7}{6}, \frac{3}{2} \end{matrix} \right]_{n}, 
\end{equation*}
 we then make use of 
\begin{equation*}
 \arcsin\left(\frac{3\sqrt{3}}{2} x(1 - x^2)\right) = 3 \arcsin\left(\frac{\sqrt{3}}{2} x\right), \quad \text{for } \frac{1}{\sqrt{3}} .\geq x 
 \geq 0, 
\end{equation*}
 noting that each side of the above equality is equal to 
\begin{equation*}
 \pi - 3 \arcsin\left(\frac{\sqrt{3}}{2} x\right), \quad \text{for } 
 1 \geq x \geq \frac{1}{\sqrt{3}}
\end{equation*}
 Setting $g = \sqrt{\frac{27}{4}}$ we obtain that 
\begin{equation*}
\int_{0}^{\frac{1}{\sqrt{3}}} \frac{3 \arcsin\left(\frac{\sqrt{3}}{2}x\right)}{x} \, dx + \int_{\frac{1}{\sqrt{3}}}^{1} \frac{\pi - 3 \arcsin\left(\frac{\sqrt{3}}{2}x\right)}{x} \, dx =2\rm{Cl}_2\big(\frac{\pi}3\big), 
\end{equation*}
 where 
\begin{equation}\label{previously5p5}
 \text{Cl}_2\big(\frac{\pi}3\big) = 
 \frac{ \sqrt{3}}{ 2} \, {}_{4}F_{3}\!\!\left[ \begin{matrix} 
 1, \frac{1}{2}, \frac{1}{2}, \frac{1}{2} \vspace{1mm}\\ 
 \frac{3}{2}, \frac{5}{6}, \frac{7}{6} \end{matrix} \ \Bigg| \ 1 \right]. 
\end{equation}
 It is also possible to find a quartic analogue using a similar factoring approach. Using \eqref{integrandforsum2}, we set 
 $a = 3$, $b = 1$, $d = 0$, $c = 0$ and $y = x^2$. We thus have that 
\begin{equation*}
 \sum_{n = 0}^{\infty} \left( \frac{27g^2}{256} \right)^{n} 
 \left[ \begin{matrix} \frac{1}{3}, \frac{1}{2}, \frac{2}{3} 
 \vspace{1mm} \\ 
 \frac{3}{4}, \frac{5}{4}, \frac{3}{2} \end{matrix} \right]_{n} 
 = 2 \int_0^1 \frac{x \left( 4x^2 - 1 \right)}{\sqrt{1 - g^2 \left( 1 - x^2 \right)^3 x^2}} \, dx. 
\end{equation*}
 Setting $g=\sqrt{\frac{256}{27}}$, we find that 
\begin{equation*}
 f(x)= \frac{2x \left( 4x^2 - 1 \right)}{\sqrt{1 - \frac{256}{27} \left(1 - x^2\right)^3 x^2}}
\end{equation*}
 assumes negative values for $ \frac{1}{2} \geq x \geq 0 $ 
 and assumes positive values for $ 1 \geq x \geq 1/2 $, so that 
\begin{align}
\begin{split}
 \int_0^1 
 \frac{\text{sgn}(4x^2 - 1) \, 
 2\sqrt{27} \, x}{\sqrt{16x^4 - 40x^2 + 27}} \, dx 
 & = \frac{3}{4} \sqrt{3} \left( 2 \ln(4 + 3\sqrt{2}) + \ln\left(\frac{7}{2} - 2\sqrt{3}\right) \right) \\ 
 & = {}_{4}F_{3}\!\!\left[ \begin{matrix} 
 1, \frac{1}{2}, \frac{2}{3}, \frac{1}{3} \vspace{1mm}\\ 
 \frac{3}{2}, \frac{5}{4}, \frac{3}{4} \end{matrix} \ \Bigg| \ 1 \right]. 
\end{split}\label{previously5p8}
\end{align}
 Little seems to be known about mixed-base ${}_{4}F_{3}$-series of elliptic type, motivating our closed-form evaluation in both 
 \eqref{previously5p5} and \eqref{previously5p8}. To the best of our knowledge, the ${}_{4}F_{3}(1)$-evaluations are original. 

\section{Hypergeometric series for $\pi^2$}
 In this section, a sextic analogue of our previous hypergeometric series is used to derive a new series {$\pi^{2}$ }. We begin 
 with the relation 
\begin{multline*}
 \int_{0}^{1} \int_{0}^{g} \frac{\arctan\left(g \sqrt{y} (1-y)\right)}{y} \, dg \, dy = \\
 \int_{0}^{1} \left( \frac{\ln\left(g^2 y (1 - y)^2 + 1\right)}{2 (y - 1) y^{3/2}} + 
 \frac{g \operatorname{arctan}\left(g (1 - y) \sqrt{y}\right)}{y} \right) \, dy, 
\end{multline*}
 and we find that the right-hand side may be written as 
 $$ \frac{2g^2}{3} \sum_{n=0}^{\infty} \left( \frac{-4g^2}{27} \right)^{n} 
 \frac{1}{n+1} 
 \left[ \begin{matrix} 1, \frac{1}{2}, \frac{1}{2} 
 \vspace{1mm} \\ 
 \frac{5}{6}, \frac{7}{6}, \frac{3}{2} \end{matrix} \right]_{n}, $$ 
 noting that 
\begin{equation*}
 \int_{0}^{1}\frac{g \, \operatorname{arctan}\left(g (1 - y) \sqrt{y}\right)}{y} \, dy=\frac{4s_1g^2}{3}. 
\end{equation*}
 Applying the substitution $y=x^2$, we find that 
\begin{equation}\label{originally6p3}
 \int_{0}^{1} \ln\left(g^2x^2 (x^2-1)^2 +1 \right) \left(\frac{1}{x^2 - 1} - \frac{1}{x^2}\right) \, dx= 
 \frac{-2g^2}{3} \sum_{n=0}^{\infty} \left( \frac{-4g^2}{27} \right)^{n} 
 \left[ \begin{matrix} \frac{1}{2}, 1, 1 
 \vspace{1mm} \\ 
 \frac{5}{6}, \frac{7}{6}, 2 \end{matrix} \right]_{n}. 
\end{equation}
 Using integration by parts, we rewrite the integral in \eqref{originally6p3} to obtain 
\begin{equation*}
\int_{0}^{1} \frac{2g^2 \left(3x^5 - 4x^3 + x\right)}{1 + g^2x^2(x^2-1)^2} \left(\frac{-1}{x} + 
 \operatorname{arctanh}(x)\right)\, dx, 
\end{equation*}
 and we proceed to use the relation such that 
\begin{equation*}
 2g^2\ \int_{0}^{1} \frac{(x^2 - 1)(3x^2 - 1)}{1 + g^2x^2(1 - x^2)^2} \, dx = 
 \frac{8g^2}{15} \sum_{n=0}^{\infty} \left( \frac{-4g^2}{27} \right)^{n} 
 \left[ \begin{matrix} \frac{1}{2}, 1 
 \vspace{1mm} \\ 
 \frac{7}{6}, \frac{11}{6} \end{matrix} \right]_{n}. 
\end{equation*}
 This may be derived by using a variant involving $\operatorname{arctan}$ of the trigonometric integral used to express $s_{2}$. 
 More specifically, our derivation is given by integrating about $g$ twice and by applying integration by parts twice, once for each 
 iteration, and then subtracting \eqref{displayw2} so as to isolate the ${}_{3}F_{2}$ that is not reducible to a polylogarithm. This 
 approach yields 
\begin{multline*}
 \frac{-2g^2}{3} \sum_{n=0}^{\infty} \left( \frac{-4g^2}{27} \right)^{n} \left[ \begin{matrix} \frac{1}{2}, 1, 1 
 \vspace{1mm} \\ 
 \frac{5}{6}, \frac{7}{6}, 2 \end{matrix} \right]_{n} = \\ 
 \int_{0}^{1} \frac{2g^2 \left(3x^5 - 4x^3 + 
 x\right)}{1 + g^2x^2(x^2-1)^2} \, \text{arctanh}(x)\, dx - 
 \frac{8g^2}{15} \sum_{n = 
 0}^{\infty} \left( \frac{-4g^2}{27} \right)^{n} 
 \left[ \begin{matrix} \frac{1}{2}, 1 
 \vspace{1mm} \\ 
 \frac{7}{6}, \frac{11}{6} \end{matrix} \right]_{n}. 
\end{multline*}
 Combining the above two summations and reversing the application of integration by parts, this yields 
\begin{equation*}
 s_4=\frac{2g^2}{15} 
 {}_{4}F_{3}\!\!\left[ \begin{matrix} 
 1, 1, 1, \frac{3}{2} \vspace{1mm}\\ 
 \frac{7}{6}, \frac{11}{6}, 2 \end{matrix} \ \Bigg| \ -\frac{4g^2}{27} \right] 
 = \int_{0}^{1} \frac{\ln\left(g^2 x^2 (x^2 - 1)^2 + 1\right)}{1 - x^2} \, dx. 
\end{equation*}
 By writing $x^2 (x^2 - 1)^2 + 1/g^2=0$ as $(x^2+r_0)(x^2+r_1)(x^2+r_2)$, 
 where $r_1=a_1+ia_2$ and $r_2=a_1-ia_2$, we obtain that 
\begin{equation*}
 s_4 = -u_1-u_2 + \frac{3\pi^{2}}{4}. 
\end{equation*}
 for 
\begin{equation*}
 \text{ { \footnotesize $ u_1=\frac{1}{4} \left( - \ln^2\left(\frac{2}{-1 + \sqrt{-r_0}}\right) 
 - \ln^2\left(\frac{-2}{1 + \sqrt{-r_0}}\right) + 2 \ln\left(\frac{1 - \sqrt{-r_0}}{2} \right) 
 \ln\left(\frac{1 + \sqrt{-r_0}}{2}\right) \right) $ } } 
\end{equation*}
 and 
\begin{equation*}
 \text{ { \footnotesize $ u_2=\frac{1}{2} \left( - \ln^2\left(\frac{2}{-1 + \sqrt{-r_1}}\right) 
 - \ln^2\left(\frac{-2}{1 + \sqrt{-r_1}}\right) + 2 \ln\left(\frac{1 - \sqrt{-r_1}}{2} \right) 
 \ln\left(\frac{1 + \sqrt{-r_1}}{2}\right) \right). $ } } 
 \end{equation*} 
 Now, setting $g := \frac{i}{u(1-u^2)}$, we obtain that $\sqrt{-r_0} = u$ and that $\sqrt{-r_1}=(-u+\sqrt{4-3u^2})/2$. After 
 simplification, setting $u=ih$, we have (for $h \geq \alpha >0 $ where $\frac{4}{27} \cdot \frac{1}{\alpha^2 (1 + \alpha^2)} = 1$):
\begin{multline*}
 \frac{4}{15h^2(1+h^2)^2} 
 {}_{4}F_{3}\!\!\left[ \begin{matrix} 
 1, 1, 1, \frac{3}{2} \vspace{1mm}\\ 
 \frac{7}{6}, \frac{11}{6}, 2 \end{matrix} \ \Bigg| \ -\frac{4}{27h^2(1+h^2)^2} \right] = \\ 
 \frac{-3}{4} \left(\pi - 2\arctan(h)\right)^2 
 + \ln^2\left(\frac{h^2 + 2 + \sqrt{3h^2 + 4}}{h\sqrt{1 + h^2}}\right). 
\end{multline*}
 Informally, the connection between $\sqrt{3}$ and $\sqrt{13}$ suggested above gives us that 
\begin{equation}\label{previously6p12}
 \ln^2\left(\frac{5 + \sqrt{13}}{2\sqrt{3}}\right) - \frac{\pi^2}{12} = 
 \frac{1}{180} {}_{4}F_{3}\!\!\left[ \begin{matrix} 
 1, 1, 1, \frac{3}{2} \vspace{1mm}\\ 
 \frac{7}{6}, \frac{11}{6}, 2 \end{matrix} \ \Bigg| \ -\frac{1}{324} \right]. 
\end{equation}
 A hypergeometric identity related to the second-to-last displayed equality given by was found by Sun and Zhou in 2024 
 \cite{SunZhoupreprint}, The closed-form evaluation in \eqref{previously6p12} appears to be new. 

\section{A proof of Campbell's conjecture} 
 To solve an open problem recently given by Campbell \cite{Campbell2025}, we begin with the following dilogarithm identity that we 
 have obtained through our beta integral-based method. 

\begin{theorem}\label{generalforCampbell}
 For complex $u$, let $\mathcal{L}(u)$ denote the three-term dilogarithm combination $$ \operatorname{Li}_{2}\left( \frac{ 2-\sqrt{4 - 
 3 u} \sqrt{u}-u }{2} \right) + \operatorname{Li}_{2}\left( \frac{2+\sqrt{4-3 u} \sqrt{u}-u}{2} \right) + 
 \operatorname{Li}_{2}(u) $$ and let $\mathcal{E}(u)$ denote the elementary expression 
\begin{multline*}
 \frac{1}{3} \ln^2\left(\frac{2}{\sqrt{4-3 u} \sqrt{u}-u} + 1\right) -\frac{1}{2} \ln ^2\left(-\frac{2}{-u+\sqrt{4-3 u} 
 \sqrt{u}+2}\right) - \\ 
 \frac{1}{2} \ln ^2\left(\frac{2}{u+\sqrt{4-3 u} \sqrt{u}-2}\right)+\frac{1}{3} \ln
 ^2\left(1-\frac{2}{u+\sqrt{4-3 u} \sqrt{u}}\right) - \\ 
 \frac{2}{3} \ln (1-u) \ln (-u) 
 + \frac{1}{3} \ln ^2(1-u) - \frac{1}{6} \ln ^2(-u) -\frac{\pi ^2}{2}. 
\end{multline*}
 For all $u \in \mathbb{C} \setminus \mathbb{R}_{>0}$, we have that 
 $\mathcal{L}(u) = \mathcal{E}(u)$. For $u \in \mathbb{R}_{>0}$, 
 we have that $\Re\left(\mathcal{L}(u) \right) = \Re\left( \mathcal{E}(u) \right)$
 and that 
 $$ \lim_{b \to 0} \Im\left(\mathcal{L}(u + bi) \right) = \lim_{b \to 0} \Im\left(\mathcal{L}(u + bi) \right). $$
\end{theorem}

\begin{proof}
 Consider the cubic integrals 
\begin{equation}\label{w1w2forCampbell}
w_1 = \int_{0}^{1} \frac{\ln\left(1 - \frac{1}{u(1-u)^2} x(1-x)^2\right)}{1-x} \, dx \text{ and } w_2 = \int_{0}^{1} \frac{\ln\left(1 - \frac{1}{u(1-u)^2} x(1-x)^2\right)}{x} \, dx 
\end{equation}
 along with 
\[ -\frac{z}{6} \sum_{n=0}^{\infty} \left( \frac{4z}{27} \right)^n 
 \left[ \begin{matrix} 
 1, 1, \frac{3}{2} 
 \vspace{1mm} \\ 
 \frac{4}{3}, \frac{5}{3}, 2 \end{matrix} \right]_{n} = 
 \int_0^1 \frac{\ln(1 - z x (1 - x)^2)}{x} \, dx. 
\]
  Using the shifting method found in the first section, we have $\frac{w_2}{w_1}=2$. Eventuating the integrals and  simplification using 
 elementary dilogarithm identities, and we obtain the desired result. 
\end{proof}

\begin{example}
 Setting $u=-1$ 
 we obtain a single-term complex dilogarithm evaluation, with 
\begin{multline*}
 \Re\mathrm{Li}_2\left(-\frac{1}{2} - \frac{i\sqrt{7}}{2}\right) 
 + \frac{\pi^2}{8} 
 + \frac{\ln^2(2)}{4} \\
 + \frac{1}{2} \left(\cot^{-1}\left(\frac{3}{\sqrt{7}}\right)\right)^2 
 - \frac{1}{3} \left(\cot^{-1}\left(\frac{5}{\sqrt{7}}\right)\right)^2 
 - \operatorname{arctan}(\sqrt{7}) \cot^{-1}\left(\frac{3}{\sqrt{7}}\right)=0. 
\end{multline*}
\end{example}

\begin{example}
 Setting $u = \frac{1}{2}$, we obtain that 
\begin{multline*}
 3 \, \mathrm{Li}_2\left(\frac{3}{4} - \frac{\sqrt{5}}{4}\right) 
 + 3 \, \mathrm{Li}_2\left(\frac{3}{4} + \frac{\sqrt{5}}{4}\right) - \\ 
 \frac{3\pi^2}{4} + \frac{3}{2} \ln^2(3 - \sqrt{5}) 
 - \ln^2(\sqrt{5} - 2) - \ln^2(2 + \sqrt{5}) + \frac{3}{2} \ln^2(3 + \sqrt{5})=0, 
\end{multline*}
 and this provides a proof of a two-term dilogarithm identity given by Adegoke and Frontczak \cite{AdegokeFrontczak2024}. 
\end{example}

\begin{example}\label{previously7p6}
 Setting $u = \frac{1}{2} + bi$, and using the reflection formula for the dilogarithm, this leads us to 
\begin{equation}\label{previously7p5}
 \Re \mathrm{Li}_2 \left(\frac{1}{2}+iu \right) = 
 \frac{{\pi}^{2}}{12} - 
 \frac{1}{8} \ln^{2}\left(\frac{1+4u^2}{4} \right) -\frac{({\arctan{(2u)})}^{2}}{2}. 
\end{equation}
 If we then plug \eqref{previously7p5} into Theorem \ref{generalforCampbell}, we obtain a two-term identity for real $b$, where 
 $r=\left(144 b^4 + 136 b^2 + 25\right)$. This gives us a closed form for 
\begin{multline*}
 \mathrm{Li}_2\left(\frac{3}{4} - \frac{2bi}{4} + \frac{r^{\frac{1}{4}}}{4} 
 \left(\cos\left(\frac{1}{2} \arctan\left(\frac{4b}{12b^2 + 5}\right)\right) 
 + i \sin\left(\frac{1}{2} \arctan\left(\frac{4b}{12b^2 + 5}\right)\right)\right)\right) + \\ 
 \mathrm{Li}_2\left(\frac{3}{4} - \frac{2bi}{4} - \frac{r^{\frac{1}{4}}}{4} 
 \left(\cos\left(\frac{1}{2} \arctan\left(\frac{4b}{12b^2 + 5}\right)\right) 
 + i \sin\left(\frac{1}{2} \arctan\left(\frac{4b}{12b^2 + 5}\right)\right)\right)\right). 
\end{multline*}
\end{example}

\begin{example}
 Setting $u = \frac{1+\sqrt{5}}{2}$, we obtain a remarkable single-term identity is derived involving the golden ratio, using Theorem 
 \ref{generalforCampbell}, namely 
\begin{equation*}
\begin{aligned}
 \frac{\pi^2}{100} = \Re \operatorname{Li}_2 \left( \frac{1}{4} \left(3 - \sqrt{5} + i \sqrt{10 - 2\sqrt{5}} \right) \right). 
\end{aligned}
\end{equation*}
 More compactly, we have that $\frac{\pi^2}{100} = 
 \Re \operatorname{Li}_2(r_o)$, where
 $ r_o \approx 0.1909... + 
 0.5877\ldots{i} $ is a root of 
\begin{equation}\label{polynomialCampbell}
 r^4 - 3r^3 + 4r^2 - 2r + 1 = 0.
\end{equation}
 This proves an identity conjectured by Campbell \cite{Campbell2025}, 
 because $1/(2\phi^2)- 1/2\sqrt{-1-1/\phi^2}$ is the same root of the polynomial in \eqref{polynomialCampbell}, 
 letting $\phi = \frac{1 + \sqrt{5}}{2}$ denote the golden ratio. 
\end{example}

\subsection{Related dilogarithm identities}
 In this section, we provide some further dilogarithm identities that we have obtained following a similar approach, relative to the 
 material given above, and we encourage further explorations based on this material. To begin with, we consider some identities that 
 are related to Theorem \ref{generalforCampbell} and that may be derived in a similar way, relative to Theorem \ref{generalforCampbell}. 

\begin{theorem}\label{previously8p1}
 The three-term dilogarithm combination 
$$ - 3\mathrm{Li}_2(u) + 3\mathrm{Li}_2\left(\frac{2}{-u - \sqrt{(4 - 3u) u} + 2}\right) 
+ 3\mathrm{Li}_2\left(\frac{2}{-u + \sqrt{(4 - 3u) u} + 2}\right) $$
 admits the closed form 
\begin{multline*}
 -\ln^2(1 - u) + \frac{1}{2} \ln^2(-u) + 2\ln(-u) \ln(1 - u) + \frac{\pi^2}{2}\\
- \ln^2\left(1 - \frac{2}{u + \sqrt{4 - 3u} \sqrt{u}}\right) 
- \ln^2\left(\frac{2}{\sqrt{4 - 3u} \sqrt{u} - u} + 1\right). 
\end{multline*}
 Equivalently, the three-term dilogarithm combination 
 $$ -3 \operatorname{Li}_2(u) 
 + 3 \operatorname{Li}_2\left(\frac{-u + \sqrt{(4 - 3 u) u} + 2}{2 (u - 1)^2}\right) 
 + 3 \operatorname{Li}_2\left(\frac{-u - \sqrt{(4 - 3 u) u} + 2}{2 (u - 1)^2}\right) $$ 
 admits the closed form 
\begin{multline*}
 - \ln^2(1 - u) + \frac{1}{2} \ln^2(-u) + 2 \ln(-u) \ln(1 - u) + \frac{\pi^2}{2} \\ 
 - \ln^2\left(\frac{2 u^2 - 3 u + \sqrt{u (4-3u)}}{2 (u - 1) u}\right) 
- \ln^2\left(\frac{2 u^2 - 3 u - \sqrt{u (4-3u)}}{2 (u - 1) u}\right). 
\end{multline*}
\end{theorem}

\begin{proof}
 This can be established following a similar approach relative to the proof of Theorem \ref{generalforCampbell}. 
\end{proof}

 This leads us toward the following result. 

\begin{theorem}\label{previouslymaster}
 For $ |u|<1$, we have that 
\begin{multline*}
 - \operatorname{Li_2} \left( \frac{2 u^2 - 3 u + \sqrt{ u (4-3u)}}{2 (u-1)^2} \right) 
 - \operatorname{Li_2} \left( \frac{2 u^2 - 3 u - \sqrt{ u (4-3u)}}{2 (u-1)^2} \right) \\ 
 = \frac{3}{2} \ln^2 \left( 1 - \frac{u - \sqrt{(4 - 3 u) u}}{2 (u-1)} \right) 
 - \operatorname{Li_2} \left( \frac{u}{u - 1} \right). 
\end{multline*}
\end{theorem}

\begin{proof}
 By Roger's five-term identity, we have that $$ w = \frac{2u^2 - 3u + \sqrt{(4 - 3u)u}}{2(u - 1)^2} \ \ \ \text{and} \ \ \ 
 z = \frac{2u^2 - 3u - \sqrt{(4 - 3u)u}}{2(u - 1)^2}. $$ By then applying the Landen identity, we obtain that 
 $$ \operatorname{Li}_2 \left(\frac{u + \sqrt{(4 - 3u) u}}{2 (u - 1)}\right) + \operatorname{Li}_2\left(\frac{u - 
 \sqrt{(4 - 3u) u}}{2 (u - 1)}\right) = -\frac{1}{2} \ln^2\left(1 - \frac{u + \sqrt{(4 - 3u) u}}{2 (u - 1)}\right). $$ 
 Similarly, for $ |u|<1$, we obtain the vanishing of 
\begin{multline*}
 \ln\left(\frac{-u + \sqrt{u (4-3u)} + 2}{2 - 2 u}\right) \ln\left(\frac{u + \sqrt{u (4-3u)} - 2}{2 (u - 1)}\right) + \\ 
 \ln^2\left(1 - \frac{u + \sqrt{(4 - 3 u) u}}{2 (u - 1)}\right) = 0. 
\end{multline*}
 This may be obtained by setting $$ a = \frac{2 - u + \sqrt{u(4 - 3u)}}{2 - 2u} \ \ \ \text{and} \ \ \ b = \frac{-2 + u + \sqrt{u(4 - 
 3u)}}{-2 + 2u}, $$ since the above equalities yield $(\ln(a)+\ln(b))^2-\ln^2(a)-\ln(a)\ln(b)=\ln^2(b) + 
 \ln(a)\ln(b)$. Because $ab=1$, we have that $(\ln(a)+\ln(b))^2=0$, so we have that $0= \ln^2(b)+\ln(a)\ln(b) + 
 \ln^2(a)+\ln(a)\ln(b)=(\ln(a)+\ln(b))^2=0$. 
\end{proof}

\begin{example}
 Setting $u=-1$ in Theorem \ref{previouslymaster}, this yields 
\begin{equation*}
\frac{\pi^2}{24} - \frac{ \ln^2(2)}{4} + \frac{1}{3} \left(\operatorname{arctan} 
 \left(\frac{\sqrt{7}}{5}\right) \right)^2 = \Re\mathrm{Li}_2 \left(\frac{3 - i \sqrt{7}}{8} \right), 
\end{equation*}
 and this evaluation appears to be original. 
\end{example}

\begin{example}
 Setting $u = \frac{1}{2} + \frac{3i\sqrt{7}}{14}$ in Theorem \ref{previouslymaster}, with $b= \frac{3\sqrt{7}}{14}$, we find that 
\begin{multline*}
 \frac{3\pi^2}{8} - \frac{\ln^2(8)}{2} + \frac{3}{8} \ln\left(\frac{49}{8}\right) \ln(8) 
 - \pi \operatorname{arctan} \left(\frac{3}{\sqrt{7}}\right) 
 + \operatorname{arctan} \left(\frac{3}{\sqrt{7}}\right)^2 + \\ 
 \operatorname{arctan} \left(\frac{\sqrt{7}}{11}\right)^2 
 + \operatorname{arctan} (\sqrt{7})^2 
 = \frac{3 }{4} \mathrm{Li}_2\left(\frac{1}{8}\right)+ 3 \, \Re \, \mathrm{Li}_2 \left(\frac{21 + i \sqrt{7}}{32} \right). 
\end{multline*}
\end{example}

\begin{example}
 Setting $u=-7$ in Theorem \ref{previouslymaster}, this yields 
\begin{equation*}
\begin{aligned}
\frac{3}{4} \operatorname{Li}_2\left(\frac{1}{8}\right)= - \frac{1}{2} \left(\operatorname{arctan} \left(\frac{5}{17 \sqrt{7}}\right)\right)^2 + \frac{3}{2}\Re\operatorname{Li}_2\left(\frac{9 - 5 i \sqrt{7}}{128}\right). 
\end{aligned}
\end{equation*}
\end{example}

\begin{example}
 Using Theorem \ref{previouslymaster}, we have also discovered the ladder relation such that 
\begin{equation}\label{previously8p8}
 \operatorname{Li}_2(r^{-6}) - 2 \operatorname{Li}_2(r^{-3}) 
 + 2 \operatorname{Li}_2\left(r^{-2}\right) = \arctan^2\left(\frac{\sqrt{7}}{11}\right) 
 - \frac{\ln^2(2)}{4}, 
\end{equation}
 writing $r=\left(\frac{1 + i \sqrt{7}}{2}\right)$. This can be obtained with the use of the relations such that 
 $$ \left( \frac{1 + i \sqrt{7}}{2} \right)^3 = -\frac{5}{2} + \frac{i \sqrt{7}}{2}, $$
 and such that 
 $$ \operatorname{Li}_2\left( \left( -\frac{5}{2} + \frac{i \sqrt{7}}{2} \right)^{-1} \right)
+ \operatorname{Li}_2\left( \left( \frac{5}{2} + \frac{i \sqrt{7}}{2} \right)^{-1} \right)
- \frac{1}{2} \operatorname{Li}_2\left( \left( \frac{1}{2}(1 + i \sqrt{7}) \right)^{-6} \right) = 0, $$
 and such that 
 $$ \operatorname{Li}_2\left( \left( \frac{5}{2} + \frac{i \sqrt{7}}{2} \right)^{-1} \right)
 = -\frac{1}{2} \ln^2\left( 1 - \left( \frac{1 + i \sqrt{7}}{2} \right)^{-2} \right) 
 - \operatorname{Li}_2\left( \left( \frac{1 + i \sqrt{7}}{2} \right)^{-2} \right). $$
\end{example}

\begin{example}
 Setting $u = \frac{3-\sqrt{5}}{2}$ in Theorem \ref{previouslymaster}, this can be used to obtain an evaluation for the two-term 
 dilogarithm combination $$ - 3 \mathrm{Li}_2\left(\frac{2 + \sqrt{5} - \sqrt{3 + 2 \sqrt{5}}}{2} \right) - 3 \mathrm{Li}_2\left(\frac{2 + 
 \sqrt{5} + \sqrt{3 + 2 \sqrt{5}}}{2} \right) $$ previously derived by Bytsko \cite{Bytsko1999}. 
\end{example}

\section{Higher-order fields}
 The algebraic arguments of the dilogarithmic expressions involved in the preceding sections were rational or in quadratic or 
 quartic algebraic number fields. This leads us to turn our attention toward higher-order algebraic number fields. 

\begin{example}
 Setting $u = r + i\sqrt{1-r^{2}}$, this leads us to apply the circle property above and \eqref{previously7p5}. Accordingly, we have 
 to solve 
\begin{equation*}
\begin{aligned}
 \frac{1}{2} = \Re \left( \frac{1}{2} \left( -i \sqrt{1 - r^2} - 
 \sqrt{(r + i \sqrt{1 - r^2}) (-3 i \sqrt{1 - r^2} - 3r + 4)} - r + 2 \right) \right). 
\end{aligned}
\end{equation*}
 The solution is given by $8 r^3 - 8 r^2 - 2 r + 3 =0$. We want the negative real root, $r=-0.57395...$, which satisfies the domain of 
 the circle property. Hence, two of the three dilogarithms of Theorem \ref{generalforCampbell} have a closed form, thus yielding a 
 single-term identity as desired. The resulting single-term identity can be expressed as the root of a solvable sextic. This have be 
 shown to produce 
 \begin{equation*}
\begin{aligned}
 -\frac{3}{4} \arctan^{2} \left( \frac{\sqrt{1 - z^2}}{z} \right) + \arctan^{2} \left( \frac{\sqrt{1 - z^2}}{z + 1} \right) 
 - \frac{1}{4} \ln^2(2(z + 1))\\ 
 + \arctan^2 \left( \frac{1}{\sqrt{8z + 3}} \right) + \frac{\pi^2}{4} - \ln^2(1 - x) = 3\Re\operatorname{Li}_2(x), 
 \end{aligned}
\end{equation*}
 where $x^6 - 2 x^5 + 13 x^4 - 24 x^3 + 19 x^2 - 7 x + 1=0$ for $x\approx 0.4495...-0.1211...i$ and $z=|r|$. 
\end{example}

\subsection{A two-term ladder identity for the Plastic Ratio/Constant}
 We will derive a two-term cubic dilogarithm ladder identity where both terms are real. These are uncommon, compared to three-term and above series or series with 2 or more complex terms. Let:
\begin{equation*}
\begin{aligned}
- \operatorname{Li}_2((1-z)^2) + \frac{\pi^2}{6} - 2 \ln(1-z) \left( \ln(2-z) + \ln z \right) = \operatorname{Li}_2(2z - z^2)\\
\operatorname{Li}_2(2-z) + \operatorname{Li}_2(z) = \frac{1}{2} \operatorname{Li}_2(2z - z^2) + \frac{\pi^2}{4}\\
\operatorname{Li}_2(z) = \operatorname{Li}_2\left(\frac{z-1}{z}\right) + \frac{1}{2} \ln^2 \left( \frac{z}{1-z} \right) - \frac{1}{2} \ln^2(1-z) + \frac{\pi^2}{6}
\end{aligned}
\end{equation*}
 Aplying elementary transformations twice, solving for $u,z$, we have that 
\begin{equation*}
\begin{aligned}
\left\{
\begin{array}{l}
z=1-\frac{1}{\frac{1}{2} \left(-u + \sqrt{(4 - 3u) u} + 2 \right)} , \\
2 - z = \frac{1}{1 - \frac{1}{2} \left(-u - \sqrt{(4 - 3u) u} + 2 \right)}
\end{array}
\right.
\end{aligned}
\end{equation*}
 We have, where $z$ is the real root of $z^3+z^2-1=0$ and $u=z$, that 
\begin{equation*}
3 \operatorname{Li}_2(z^2) + 6 \operatorname{Li}_2(z) = - \ln^2(z - z^2) - 6 \ln(z) \ln(z + 1) + \pi^2
\end{equation*}
 Note that $z=1/p$ where $p$ is the so-called \textit{Plastic Constant} or ratio, given as the real root of $p^3-p-1=0$. It also obeys 
 the following identities: $p-1=p^{-4}, p^{-5}=1-1/p,p^{3}=p+1,p^{2}=1+1/p$. Thus we have the following relations: 
\begin{equation*}
\begin{aligned}
\mathrm{Li_2}(p^{-2}) + 2\mathrm{Li_2}(p^{-1}) &= -8\ln^2(p) + \frac{\pi^2}{3} \\
\mathrm{Li_2}(p^{-1}) + \mathrm{Li_2}(p^{-5}) &= \frac{\pi^2}{6} - 5\ln^2(p) \\ 
\mathrm{Li_2}(p^{-2}) - 2\mathrm{Li_2}(p^{-5}) &= 2\ln^2(p)\\
\operatorname{Li}_2(p^2) + \operatorname{Li}_2(p^3) = -\frac{1}{2} \ln^2(p) + \frac{\pi^2}{2}\\
\operatorname{Li}_2(-p^{-4}) - \operatorname{Li}_2(p^{-1}) = -\frac{\pi^2}{6} + \frac{9}{2} \ln^2(p)
\end{aligned}
\end{equation*}
 This also leads to various ladder expressions: 
\begin{equation*}
\begin{aligned}
\frac{5}{2} \mathrm{Li}_2(p^{-2}) + \mathrm{Li}_2(p^{-1}) - 4\mathrm{Li}_2(p^{-5}) = \frac{\pi^2}{6}\\
\mathrm{Li}_2\left(\frac{1}{p^{14}}\right) - 4\mathrm{Li}_2\left(\frac{1}{p^7}\right) + 6\mathrm{Li}_2\left(\frac{1}{p^5}\right) = \frac{\pi^2}{6}-8\ln^2(p)
\end{aligned}
\end{equation*}

\begin{example}
 Setting $u=1/2+i\mu$, it is possible to derive additional single-term dilogarithms by combining the first identity in Theorem 
 \ref{previously8p1} with \eqref{previously7p5}. Define $\mu$ as an algebraic number given by the expression below. This is derived 
 by setting solving for real $\mu$ for $u=1/2+i \mu$ , which has four possible solutions: 
\begin{equation*}
 \frac{1}{2} = \Re\left(\frac{4}{\pm\sqrt{12 \mu^2 + 4 i \mu + 5} - 2 i \mu + 3}\right) 
\end{equation*}
 The solutions are the real roots of $-256 c^8 + 96 c^4 + 544 c^2 - 125=0$, or: 
\begin{equation*}
\mu_1\approx\pm1.1441...\mu_2 \approx \pm 0.4714... 
\end{equation*}
 Plugging $u=1/2+\mu_2i$ into the first identity in Theorem \ref{previously8p1} thus produces a single-term dilogarithm identity, and 
 the same holds for $u=1/2+\mu_1i$. Moreover, we have that $\mu_1,\mu_2$ admit closed forms, with $$ \mu_1^2=\frac{1}{32} 
 \sqrt{\frac{k_1+192}{3}} + \frac{1}{2} \sqrt{-\frac{k_1}{768}+\frac{1}{2} + 68 \sqrt{\frac{3}{k_1+192}}} $$ and $$ \mu_2^2 = 
 \frac{1}{32} \sqrt{\frac{k_1+192}{3}} - \frac{1}{2} \sqrt{-\frac{k_1}{768}+\frac{1}{2} + 68 \sqrt{\frac{3}{k_1+192}}}, $$ writing $$ 
 k_1=128 \cdot 2^{2/3} \cdot 3^{1/3} \left( (45 - \sqrt{1929})^{1/3} + (45 + \sqrt{1929})^{1/3} \right). $$ Using the two-term identity in 
 Example \ref{previously7p6}, these are obtained by solving a system of equations, where $\theta=1$ corresponds to $\mu_1$ and $ 
 \theta=-1$ corresponds to $\mu_2$. 
\begin{equation*}
\begin{aligned}
 \frac{3}{4} + \theta \frac{\sqrt{12 \mu^2 + 5 + \sqrt{144 \mu^4 + 136 \mu^2 + 25}}}{4 \sqrt{2}} = \frac{2}{4k^2+1} \\ 
 - \frac{\mu}{2}+\theta \frac{\sqrt{-12 \mu^2 -5 + \sqrt{144 \mu^4 + 136 \mu^2 + 25}}}{4 \sqrt{2}} = \frac{-4k}{4k^2+1} 
\end{aligned}
\end{equation*}
 This allows us to produce the explicit identity such that 
\begin{multline*}
 -3 \Re\operatorname{Li}_2 \left( \frac{24 k^2 - 
 2 + 4 i (4 k^2 \mu - 4 k + \mu)}{4 k^2 (4 \mu^2 + 9) - 32 k \mu + 4 \mu^2 + 1} \right) = 
 -\frac{1}{8} \ln^2(k^2 + \frac{1}{4}) + \\ 
 \ln^2 \left( \frac{4 k \mu + 6 i k - 2 i \mu - 1}{4 k \mu + 2 i k - 2 i \mu - 3} \right) 
- \frac{5}{2} \arctan^{2}(2 k) 
- \arctan^{2}(2 \mu) 
+ \pi \arctan(2 \mu). 
\end{multline*}
\end{example}

\section{Applications of Roger's identity}
 Our below applications of Roger's identity are useful in the construction of new dilogarithm ladder relations, as we later demonstrate. 

\begin{lemma}\label{previously10p1}
 The equality of 
\begin{multline*}
 3 \operatorname{Li}_2\left(\frac{-2u^2 + 3u + \sqrt{-u (3u - 4)}}{4u - 2u^2}\right) + 
 3 \operatorname{Li}_2\left(\frac{-2u^2 + 3u - \sqrt{-u (3u - 4)}}{4u - 2u^2}\right) + \\ 
 3 \operatorname{Li}_2(u) + 3 \operatorname{Li}_2((u - 1)^2) 
\end{multline*}
 and 
\begin{multline*}
 \text{ {\footnotesize $ - \frac{3}{2} \ln^2\left(\frac{2}{u - \sqrt{(4 - 3u) u} - 2}\right) 
 - \frac{3}{2} \ln^2\left(\frac{2}{u + \sqrt{(4 - 3u) u} - 2}\right) 
 + \ln^2\left(1 - \frac{2}{u + \sqrt{(4 - 3u) u}}\right) $ } } \\ 
 \text{ { \footnotesize $ + \ln^2\left(1 + \frac{2}{\sqrt{(4 - 3u) u} - u}\right) 
 - 3 \ln\left(\frac{u - \sqrt{-u (3u - 4)}}{4u - 2u^2}\right) 
 \ln\left(\frac{u + \sqrt{-u (3u - 4)}}{4u - 2u^2}\right) + $ } } \\ 
 \text{ { \footnotesize $ \ln^2(1 - u) - \frac{\ln^2(-u)}{2} - 2 \ln(-u) \ln(1 - u) - \frac{3 \pi^2 }{2}$. } } 
 \end{multline*}
 holds for complex $u$. 
\end{lemma}

\begin{proof}
 By setting $$w = \frac{-u - \sqrt{(4 - 3u)u} + 2}{2} \ \ \ \text{and} \ \ \ z = \frac{-u + \sqrt{(4 - 3u)u} + 2}{2}, $$ the desired result 
 follows from the application of Roger's identity to Theorem \ref{generalforCampbell}. 
\end{proof}

\begin{lemma}\label{previously10p2}
 The equality of 
\begin{multline*}
 \operatorname{Li}_2 \left( \frac{2 u^2 - 3 u - \sqrt{u(4-3u)}}{2 (u-1)^2} \right) 
 + \operatorname{Li}_2 \left( \frac{(u - 1) \left(2 u^3 - 4 u^2 + 3 u - \sqrt{u(4 - 
 3 u)} \right)}{2 \left(u^4 - 3 u^3 + 4 u^2 - 2 u + 1\right)} \right) \\ 
 + \operatorname{Li}_2 \left( \frac{u \left( u^2 - u + (u - 1) \sqrt{u(4-3u)} + 
 2 \right)}{2 \left(u^4 - 3 u^3 + 4 u^2 - 2 u + 1\right)} \right) 
 + \operatorname{Li}_2 \left( \frac{u \left(2 u^2 - 3 u + \sqrt{u(4-3u)}\right)}{2 (u-1)^2} \right) 
\end{multline*}
 and the closed form 
\begin{multline*}
 \text{ { \footnotesize $ - \ln \left( \frac{u^2 - u + (u - 
 1) \sqrt{u(4 - 3 u)} + 2}{2 \left(u^4 - 3 u^3 + 4 u^2 - 2 u + 1\right)} \right) 
 \ln \left( \frac{(u - 1) \left(2 u^3 - 5 u^2 - u \sqrt{(4 - 3 u) u} + 4 u - 2\right)}{2 \left(u^4 - 3 u^3 + 
 4 u^2 - 2 u + 1\right)} \right) - $ } } \\ 
 \text{ { \footnotesize $ \frac{3}{2} \ln^2 \left( 1 - \frac{u - \sqrt{u(4-3u)}}{2(u-1)} \right) - 
 \frac{1}{2} \ln^2 (1 - u)$. } }
\end{multline*}
 holds for complex $u$. 
\end{lemma}

\begin{proof}
 Setting $w = u$ and $$ z = \frac{2u^2-3u \pm \sqrt{(4 - 3u)u} + 2}{2(u-1)^2}, $$ the desired result follows from the application 
 of Roger's identity to Theorem \ref{previouslymaster}. 
\end{proof}

\begin{lemma}\label{previously10p3}
 The equality of 
\begin{multline*}
 \operatorname{Li}_2\left(\frac{u - \sqrt{u(4 - 3u)}}{4u - 2u^2}\right) + 
 \operatorname{Li}_2\left(\frac{u + \sqrt{u(4 - 3u)}}{4u - 2u^2}\right) + \\ 
 \operatorname{Li}_2\left(\frac{1}{(u - 1)^2}\right) + \operatorname{Li}_2\left(\frac{u}{u - 1}\right) - \frac{\pi^2}{3} 
\end{multline*}
 and the closed form 
\begin{multline*}
 \text{ { \footnotesize $ \frac{3}{2} \ln^2 \left(\frac{u - \sqrt{-u (3 u - 4)} - 2}{2 (u - 1)}\right) 
 - \ln \left(\frac{-2 u^2 + 3 u + \sqrt{-u (3 u - 4)}}{4 u - 2 u^2}\right) \ln \left(\frac{2 
 u^2 - 3 u + \sqrt{-u (3 u - 4)}}{2 (u - 2) u}\right) $ } } \\ 
 \text{ { \footnotesize $ - \ln \left(\frac{-u + \sqrt{(4-3 u) u} + 2}{2 (u - 
 1)^2}\right) \ln \left(\frac{2 u^2 - 3 u - \sqrt{-u (3u-4)}}{2(u-1)^2}\right) - $}} \\ 
 \text{ { \footnotesize $ \ln \left(\frac{-u - \sqrt{(4 - 3 u) u} + 2}{2 (u - 1)^2}\right) \ln \left(\frac{2 
 u^2 - 3 u + \sqrt{-u (3 u - 4)}}{2(u-1)^2}\right)$ } }
\end{multline*}
 holds for complex $u$. 
\end{lemma}

\begin{proof}
 Setting $$w = \frac{-u + \sqrt{-u (3 u - 4)} + 2}{2 (u - 1)^2} \ \ \ \text{and} \ \ \ z = \frac{-u - \sqrt{-u (3 u - 4)} + 2}{2 (u - 
 1)^2}, $$ the desired result follows through an application of Euler's identity and then Roger's identity to 
 Theorem \ref{previouslymaster}. 
\end{proof}

\begin{lemma}\label{previously10p4}
 The equality of 
\begin{multline*}
 \operatorname{Li}_2(k) - \operatorname{Li}_2(k^2) 
 - \operatorname{Li}_2\left(\frac{-2k^2 + k - \sqrt{-3k^2 + 2k + 1} + 1}{2 - 2k^2}\right) - \\ 
 \operatorname{Li}_2\left(\frac{-2k^2 + k + \sqrt{-3k^2 + 2k + 1} + 1}{2 - 2k^2}\right) 
\end{multline*}
 and the closed form 
\begin{multline*}
\frac{1}{2} \ln^2\left(\frac{-2}{-\sqrt{-3k^2 + 2k + 1} + k + 1}\right) 
+ \frac{1}{2} \ln^2\left(\frac{-2}{\sqrt{-3k^2 + 2k + 1} + k + 1}\right)\\ 
- \frac{1}{3} \ln^2\left(1 - \frac{2}{\sqrt{-3k^2 + 2k + 1} - k + 1}\right)
- \frac{1}{3} \ln^2\left(\frac{2}{\sqrt{-3k^2 + 2k + 1} + k - 1} + 1\right) \\
+ \ln\left(\frac{\sqrt{-3k^2 + 2k + 1} - k + 1}{2 - 2k^2}\right) 
\ln\left(\frac{\sqrt{-3k^2 + 2k + 1} + k - 1}{2k^2 - 2}\right)\\
+ \frac{1}{2} \pi^2 + \frac{1}{6} \ln^2(1 - k) - \frac{1}{3} \ln(1 - k) \ln k - \frac{1}{3} \ln^2 k
\end{multline*}
 holds for complex $k$. 
\end{lemma}

\begin{proof}
 This follows from the application of Roger's identity to Theorem \ref{generalforCampbell}, with $$w=\left(\frac{-u + 
 \sqrt{(4 - 3u)u} + 2}{2}\right) \ \ \ 
 \text{and} \ \ \ 
 z=\left(\frac{-u - \sqrt{(4 - 3u)u} + 2}{2}\right), $$ 
 and by enforcing the substituion $u=1-k$. 
\end{proof}

\begin{lemma}\label{previously10p5}
 The equality of
\begin{multline*}
 - \operatorname{Li}_2 \left( (u - 1) u \right) 
 + \operatorname{Li}_2 \left( \frac{-u^3 + 5u^2 + \sqrt{u(4 - 3u)} (u^2 - u - 1) - 5u}{2(u - 1)^2} \right) 
 - \operatorname{Li}_2 \left( \frac{u}{u - 1} \right)\\ 
 + \operatorname{Li}_2 \left( \frac{-u^3 + 5u^2 - \sqrt{u(4 - 3u)} (u^2 - u - 1) - 5u}{2(u - 1)^2} \right) 
\end{multline*}
 and the closed form 
\begin{multline*}
 - \frac{3}{2} \ln^2 \left( 1 - \frac{u - \sqrt{u(4 - 3u)}}{2(u - 1)} \right)\\ 
 + \ln \left( \frac{-u + \sqrt{-u(3u - 4)} + 2}{2(u - 1)^2} \right) 
 \ln \left( \frac{-(u + \sqrt{-u(3u - 4)} - 2)}{2(u - 1)^2} \right) 
\end{multline*}
 holds for complex $u$. 
\end{lemma}

\begin{proof}
 This follows form the application of Abel's five-term identity to Theorem \ref{previouslymaster} 
 where $w = \frac{-u + \sqrt{-u (3 u - 4)} + 2}{2 (u - 1)^2}, z = \frac{-u - \sqrt{-u (3 u - 4)} + 2}{2 (u - 1)^2}$. 
\end{proof}

\section{Derivation of the Loxton--Lewin $\frac{\pi}{9}$-identities}
 In this section, we derive an elementary proof to the trio of three and four-term Loxton-Lewin dilogarithm $\frac{\pi}{9}$-identities. For reference, the three-term is identity given by, where $h = \frac{1}{2} \sec\left(\frac{\pi}{9}\right)$:
\begin{equation*}
\begin{aligned}
 \left( \operatorname{Li}_2(h) + \frac{1}{2} \ln(h) \ln(1 - h) \right) 
 + \left( \operatorname{Li}_2(h^2) + \frac{1}{2} \ln(h^2) \ln(1 - h^2) \right) - \\ 
 \frac{1}{3} \left( \operatorname{Li}_2(h^3) + \frac{1}{2} \ln(h^3) \ln(1 - h^3) \right)=7\pi^2/54, 
\end{aligned}
\end{equation*}
 Using the latter identity in Theorem \ref{previously8p1} and applying the Euler reflection identity in \eqref{Eulerreflectionidentity} 
 and then letting $u=1-k$, we have to show that the expression below yields the sought three-term relation in terms of $h$: 
\begin{multline*}
 - \operatorname{Li}_2\left(\frac{2 k^2 - k - \sqrt{-3 k^2 + 2 k + 1} - 1}{2 k^2}\right) - \\ 
 \operatorname{Li}_2\left(\frac{2 k^2 - k + \sqrt{-3 k^2 + 2 k + 1} - 1}{2 k^2}\right) 
 - \operatorname{Li}_2(1-k). 
\end{multline*}
 We then set 
\begin{equation*}
\begin{aligned}
\frac{\sqrt{-3 k^2 + 2 k + 1}}{2 k^2 - k - 1} = -i \sqrt{3}
\end{aligned}
\end{equation*}
 This gives the polynomial $3 k^3 - 3 k^1 - 1 =0$, which has the solution: $k=\frac{-2}{\sqrt{3}} \sin\left(\frac{\pi}{9}\right)$. 

 With reference to the polylogarithm multiplication formula 
\begin{equation}\label{previously11p4}
\begin{aligned}
 \sum_{m=0}^{p-1} \operatorname{Li}_s(z e^{2\pi i m/p}) = p^{1-s} \operatorname{Li}_s(z^p), 
\end{aligned}
\end{equation}
 we may obtain that 
\begin{equation}\label{previously11p5}
\begin{aligned}
 -\operatorname{Li}_2(r) + \frac{1}{3} \operatorname{Li}_2(r^3) = 
 \operatorname{Li}_2\left(r e^{2\pi i / 3}\right)+\operatorname{Li}_2\left(r e^{4\pi i / 3}\right) 
\end{aligned}
\end{equation}
 So now we derive $r$, which is the radius/norm, given by: 
\begin{equation*}
\frac{\sqrt{(2 k^2 - k - 1)^2 - (-3 k^2 + 2 k + 1)}}{2 k^2} = \sqrt{\frac{k - 1}{k}}=\frac{1}{h}
\end{equation*}
 To prove this, make a substitution $3 (1/(1-x^2))^3 - 3 (1/(1-x^2))^1 - 1 =0$. Here, $k= (1/(1-x^2))$ This factors
 as $(x^3 - 3 x - 1) (x^3 - 3 x + 1) = 0$. We choose the polynomial $(x^3 - 3 x - 1) $, which has a root $1.87...$ which is $1/h$. 
 Using the dilogarithm relation in \eqref{sigmasumLi}, we find that 
\begin{equation}\label{previously11p7}
 - \operatorname{Li}_2\left(\frac{2 k^2 - k- \sqrt{-3 k^2 + 2 k + 1} + 1}{2 k^2}\right) 
 - \operatorname{Li}_2\left(\frac{2 k^2 - k + \sqrt{-3 k^2 + 2 k + 1} - 1}{2 k^2}\right) 
\end{equation}
 reduces to $$ -\frac{1}{3} \operatorname{Li}_2(h^{-3}) + \operatorname{Li}_2(h^{-1}).$$ We then use the Landen identity in 
 \eqref{Landenformula}, with $\operatorname{Li}_2(1-k)$. By inspection, $\frac{k-1}{k} = h^{-2}$, which is the squared term of the 
 ladder, completing the proof, in so far that we have proven the existence of the identity. Additional simplifications of log terms are 
 required to bring it to the familiar form. This is done in the second proof below. 

\subsection{Derivation of the Loxton-Lewin four-term $\frac{\pi}{9}$-identities}
 To derive a four-term dilogarithm involving $2\cos\big(\frac{4\pi}{9}\big)$ as the base, we apply Theorem 
 \ref{previouslymaster} and set 
\begin{equation*}
 \frac{\sqrt{u (4-3u)}}{2 u^2 - 3 u} = i \sqrt{3}. 
\end{equation*}
 This has the solution $u=1 - \frac{2}{\sqrt{3}} \sin\left(\frac{4\pi}{9}\right)$ 
 within the unit circle, so Theorem \ref{previouslymaster} can be applied. The radius, in this case, is 
 $r=\sqrt{\frac{u}{u - 1}} = 2\cos\left(\frac{4\pi}{9}\right)$. 
 Consequently, we obtain that 
\begin{equation*}
 \frac{3}{2} \ln^2 \left( 1 - \frac{u - \sqrt{(4 - 3 u) u}}{2 (u-1)} \right) 
 - \operatorname{Li_2} (r^2) = \operatorname{Li}_2 (-r) - \frac{1}{3} \operatorname{Li}_2 (-r^3). 
\end{equation*}
 For the logarithm on the left-hand side, split it into the complex and real parts. (Note that because $u=-.137...$ is negative that 
 $\sqrt{(4 - 3 u)u}$ is a complex number.) Also note that $\sqrt{1 - \left( \frac{2(u - 1) - u}{2(u - 1)} \right)^2} = \sqrt{1 - 
 \left(\frac{(u - 2)}{2 (u - 1)}\right)^2} = \pm \frac{\sqrt{3u^2 - 4u}}{2(u - 1)}$, so therefore given the constraints that it is a primitive root 
 of unity. Performing the same process as in \eqref{previously11p4}, we observe that $\frac{(u - 2)}{2 (u - 1)}$ is a root of $8 x^3 - 6 x - 
 1=0$, or exactly $\cos(\frac{\pi}{9})$. So we have $3/2 \ln^2(e^{-i\frac{\pi}{9}})$ (the negative $-\pi i/9$ sign arises because we choose 
 the positive value of $\pm \frac{\sqrt{3u^2 - 4u}}{2(u - 1)}$ , as $u-1$ is negative), which gives the sought $-\pi^2/54$ value. 
 Using the duplication formula for $\operatorname{Li_2}$
 and elementary manipulations, we obtain the desired final form. 

 Mappings of the form $\frac{u}{u-1} \mapsto \frac{a+bx}{cx+d}$ provide all possible solutions that satisfy $(x^3 - 3 x - 1) (x^3 - 
 3 x + 1) = 0$. For our purposes, we only require that the ladder is composed of roots of the above polynomial, provided valid solutions 
 are chosen for Roger's identity, as not all will be applicable. For the third identity, the solution $2\cos\big(4\frac{\pi}{9} \big)$ is chosen. 
 It just so happens the solutions are trigonometric values. 

 Lewin and Loxton sought an analytic proof to the third and most difficult of the $\frac{\pi}{9}$-identities. It appears that the 
 above proof is the closest proof for this identity in terms of being closest to a purely analytic proof. In contrast, a proof due to Kirillov 
 \cite{Kirillov1995} relies on the assembly of the identity from various relations of trigonometric components and elementary 
 dilogarithm identities, whereas the radius method sidesteps this assembly process. 

 To derive the second four-term ladder that involves $\frac{1}{2} \sec(\frac{2\pi}{9})$, we again apply the latter identity in Theorem 
 \ref{previously8p1} and applying the Euler reflection identity, and we again set $u = 1 -k$, but now we set 
\begin{equation*}
\begin{aligned}
\frac{\sqrt{-3 k^2 + 2 k + 1}}{2 k^2 - k - 1} = i \sqrt{3}. 
\end{aligned}
\end{equation*}
 The solution $k=\frac{-2}{\sqrt{3}} \sin\left(\frac{2\pi}{9}\right)$ and the radius is $r=1+\frac{1}{2} \sec\left(\frac{\pi}{9}\right)$. We can 
 show that $\frac{1}{r} = \frac{1}{2} \sec\left(\frac{2\pi}{9}\right)$, by mimicking the derivation of \eqref{previously11p5}, and by 
 changing the sign from $h$ to $-h$, completes the proof. By analogy with the reduction of \eqref{previously11p7}, the additional 
 transformation $\operatorname{Li_2}(z) + \operatorname{Li_2}(-z)=1/2\operatorname{Li_2}(z^2)$ is needed to convert it to the 
 four-term ladder form. 

 Alternatively, all three can be proven at once with Roger's five-term identity, with the following substitutions, which are chosen 
 in such a way that Landen's identity can be applied: 
\begin{equation*}
 n = \frac{\sqrt{-3a^2 + 2a + 1} + a + 1}{2a},z = -\frac{n}{(n - 1)^2}, \quad w = n - n^2. 
\end{equation*}
 Written out, this the equality of the three-term dilogarithm combination 
\begin{multline*}
 - \operatorname{Li}_2 \left( \frac{a}{a - 1} \right) 
 - \operatorname{Li}_2 \left( \frac{2a^2 - a - \sqrt{-3a^2 + 2a + 1} - 1}{2a^2} \right) \\ 
 - \operatorname{Li}_2 \left( \frac{2a^2 - a + \sqrt{2a - 3a^2 + 1} - 1}{2a^2} \right) 
\end{multline*}
 and the closed form 
\begin{multline*}
 \frac{\pi^2}{6} 
 + \frac{1}{2} \ln^2 \left( \frac{-2a^2 + \sqrt{-3a^2 + 2a + 1} + a + 1}{2(a - 1)a} \right) + \\ 
 \frac{3}{2} \ln^2 \left( 1 - \frac{\sqrt{-3a^2 + 2a + 1} + a + 1}{2a} \right). 
\end{multline*}
 By setting 
\begin{equation*}
 \frac{\sqrt{(2 - 3a)a + 1}}{a(2a - 1) - 1} = \pm\sqrt{3}i, 
\end{equation*}
 we obtain the radius $\sqrt{\frac{a - 1}{a}}$. 

 A third proof relies on the application of Hill's five-term identity 
\begin{equation*}
\operatorname{Li}_2(zw) - \operatorname{Li}_2(w) - \operatorname{Li}_2(z) - \operatorname{Li}_2\left(\frac{z(w - 1)}{1 - z}\right) - \operatorname{Li}_2\left(\frac{w(z - 1)}{1 - w}\right) - \frac{1}{2} \ln^2\left(\frac{1 - w}{1 - z}\right) = 0
\end{equation*}
 Letting $w=r e^{\frac{2\pi i}{3}},z=r e^{\frac{4\pi i}{3}}$ or $w=r e^{\frac{\pi i}{3}},z=r e^{\frac{5\pi i}{3}}$. 

 After simplifying, note: $\frac{r (r^2 + 4r + 1)}{2 (r^2 + r + 1)} \pm \frac{i \sqrt{3} \, r (r^2 - 1)}{2 (r^2 + r + 1)}$. 
 So, we need to solve 
\begin{equation*}
 \frac{r (r^2 + 4r + 1)}{2 (r^2 + r + 1)} = \frac{1}{2}
\end{equation*}

 The solutions are $r=\frac{1}{2} \sec\left(\frac{\pi}{9}\right), \quad \frac{-1}{2} \sec\left(\frac{2\pi}{9}\right), \quad 
 \frac{-1}{2} \sec\left(\frac{4\pi}{9}\right)$. Then we use \eqref{previously7p5} to find the closed-form value of the real part, which when 
 combined with additional simplifications, thus gives the value of $k$ and proving the identities, using \eqref{previously11p5}. 
 Note: $\frac{1}{2} \sec\left(\frac{4\pi}{9}\right)$ will not work, as it is outside of the convergence. 

 Observe that the
 latter two proofs use five-term identities once, whereas the first proof uses the 
 cubic integral. As a consequence of the connection between the integral, the hypergeometric series, and the dilogarithm identity, 
 we also have that 
\begin{equation*}
 \cos\left(\frac{\pi}{9}\right) \sum_{n=0}^{\infty} \left( \frac{-8\cos\left(\frac{\pi}{9}\right)}{9} \right)^n 
 \frac{ \left( \frac{3}{2} \right)_n (1)_n^2 }{ \left( \frac{5}{3} \right)_n \left( \frac{4}{3} \right)_n 
 (2)_n } = \ln^2\left(2 \cos\left(\frac{4\pi}{9}\right)\right) - \frac{\pi^2}{27} 
 \end{equation*} 
 The $2\cos(\frac{4\pi}{9})$ identity is regarded as the most difficult of the three. In this case, this identity is the easiest, as it requires the 
 least simplification of log terms due to $u$ lying within the unit circle, so Theorem \ref{previouslymaster} can be immediately used. 

 \section{Two-term complex identities of the imaginary part} 
 Similar to two-term real dilogarithms, two-term dilogarithms of only the imaginary part appear to be uncommon. Common examples 
 of two-term identities of the complex part include 
 $\operatorname{Li}_2\left(\frac{1}{2} - \frac{i}{2}\right) = -i G + \frac{5\pi^2}{96} - \frac{\ln^2 2}{8} + \frac{i \pi \ln 2}{8}$. 

\begin{example}
 Using Theorem \ref{previouslymaster}, set 
\begin{equation*}
 \frac{2 u^2 - 3 u - \sqrt{u (-3 u + 4)}}{2 (u - 1)^2} = -u 
\end{equation*}
  The solution is given by $u^3-u-1=0$ (the Plastic constant cubic as before), where the complex root in the second   quadrant is chosen.  
 Performing the usual simplification gives:  
\begin{equation*}
  -4i \pi \ln(u) + \frac{5 \ln^2(u)}{2} - \frac{5\pi^2}{3} = \operatorname{Li}_2(u) + \frac{1}{2} \operatorname{Li}_2(u^2)  
\end{equation*}
  Here, the duplication formula for the dilogarithm is implicitly used.  
\end{example}

\begin{example}
 Using Theorem \ref{previouslymaster}, set:
 \begin{equation*}
\frac{2 u^2 - 3 u - \sqrt{u (-3 u + 4)}}{2 (u - 1)^2} = 2 - u
\end{equation*}
 The solution is given by 
 $u^4 - 4 u^3 + 7 u^2 - 7 u +4=0$ where $u\approx0.45258 - 1.12087 i$. After much simplification, we have:
\begin{equation*}
 \frac{1}{2} \operatorname{Li}_2(2u - u^2) + \operatorname{Li}_2\left(\frac{u - 2}{(u - 1)^2} + u\right) = -\frac{13}{2} \ln^2(u - 1) - 7 i \pi \ln(u - 1) - i \pi \ln(2 - u) + \frac{7\pi^2}{4}
\end{equation*}
\end{example}

\begin{example}
 Using Lemma \ref{previously10p3}, set: 
\begin{equation*}
 \displaystyle \frac{\sqrt{u(4 - 3u)} + u}{4u - 2u^2} = \frac{1}{z}, \\[10pt] 
 \displaystyle \frac{1}{(u - 1)^2} = z 
\end{equation*}
 This has the solution: $u^6 - 7 u^5 + 19 u^4 - 24 u^3 + 13 u^2 - 2 u + 1=0$ in which the root approximated 
 by $u= -0.00219275 - 0.295542i$ is chosen. The identity follows (using the Bloch-Wigner formula for brevity), 
 where $a_1=\left(\frac{u^3 - 4u^2 + 5u - 1}{2 - u}\right), a_2=\left(\frac{u}{u - 1}\right)$. Thus: 
\begin{equation*}
 -\arg(1-a_1)\ln(|a_1|)-\arg(1-a_2)\ln(|a_2|)=\Im\operatorname{Li}_2(a_1) + \Im\operatorname{Li}_2(a_2)
\end{equation*}
\end{example}

\begin{example}
 Using Lemma \ref{previously10p2}, set: 
\begin{equation*}
\left\{
\begin{aligned}
\frac{(u-1) \left(2 u^3 - 4 u^2 + 3 u - \sqrt{u(4 - 3 u)}\right)}{2 \left(u^4 - 3 u^3 + 4 u^2 - 2 u + 1\right)} = z\\
\frac{u \left(u^2 - u + (u-1) \sqrt{u(4-3u)} + 2\right)}{2 \left(u^4 - 3 u^3 + 4 u^2 - 2 u + 1\right)} = \frac{1}{z}
\end{aligned}
\right.
\end{equation*}
 Using the above procedure, we have that 
\begin{equation*}
 - \Im \operatorname{Li}_2(v_1) - \Im \operatorname{Li}_2(v_2) = \arg(1 - v_1) \ln(|v_1|) + \arg(1 - v_2) \ln(|v_2|). 
\end{equation*}
 Applying $u=1-y$ to the final result, we obtain 
\begin{equation*}
v_1=\left( \frac{-(y - 1)^2 (y^3 + y^2 + 1)}{y (y^2 - y + 1)} \right),
v_2=\left( \frac{-(y - 1)^2 (y^3 + y + 1))}{y^2 (y^2 - y + 1)} \right), 
\end{equation*}
 where $y \approx 0.18866 - 0.83324 i$ is a root of $ y^8 - y^7 + y^6 - y^4 + y^2 - y + 1 = 0$. We conjecture that this is the only 
 8th-degree, two-term dilogarithm combination that admits a closed form and that it is the highest degree identity of this form. 
\end{example}

\section{Miscellaneous results} 
 Through the use of the Landen identity in \eqref{Landenformula}, and by solving for $u$ and $z$, this leads to: 
\begin{equation*}
\begin{aligned}
\frac{-2u^2 + 3u + \sqrt{-u (3u - 4)}}{4u - 2u^2} = z,
\frac{-2u^2 + 3u - \sqrt{-u (3u - 4)}}{4u - 2u^2} = \frac{z}{z - 1}
\end{aligned}
\end{equation*}
 The solutions to $u^3 - 4 u^2 + 3 u + 1 = 0$ correspond to the same root structure as trigonometric values of multiples of 
 $\frac{\pi}{7}$. The valid solution $u=2\cos (\frac{3\pi}{7})+1$ proves one of
 the Watson two-term identities through some additional manipulations.

\subsection{A cubic single-term identity}
Using Lemma \ref{previously10p2}, set:
\begin{equation*}
\begin{aligned}
\frac{(u-1) \left(2 u^3 - 4 u^2 + 3 u - \sqrt{u(4 - 3 u)}\right)}{2 \left(u^4 - 3 u^3 + 4 u^2 - 2 u + 1\right)} = z
\\
\frac{2 u^2 - 3 u - \sqrt{u(4 - 3 u)}}{2 (u-1)^2} = 1 - z
\end{aligned}
\end{equation*}
 The solution $u=r_1$ is the real root of $u^3 - 2 u^2 + u + 1=0$. And $z=1/r_2$, where $r_2$ is the complex root located 
 in the fourth quadrant.
 Consequently, we obtain that 
\begin{equation*}
\begin{aligned}
2 \ln\left(\frac{r_2 - 1}{r_2}\right) \ln\left(\frac{(r_1 - 1) r_2^2}{2 r_1 r_2 + r_1}\right) 
+ 3 \ln^2\left(\frac{1 - r_1}{r_2}\right) 
+ \ln^2(1 - r_1) 
+ \frac{\pi^2}{3} 
= -4 \Re \operatorname{Li}_2\left(r_1 - \frac{r_1}{r_2}\right)
\end{aligned}
\end{equation*}
 Using Lemma \ref{previously10p1} and setting $u=p^3$, where $p$ is the 
 real root of $u^3-u-1=0$ also gives the single-term dilogarithm in terms of the plastic constant/ratio: 
\begin{equation*}
\begin{aligned}
\frac{\pi^2}{12} + 6 \ln^2(p) + \frac{1}{6} \left( \operatorname{arctan} \left( \frac{\sqrt{p (3 p - 1)}}{p^2 - 2 p^3} \right) \right)^2 
= \Re\mathrm{Li}_2\left( \frac{-p^4 + p^2 i \sqrt{p (3p - 1)}}{2} \right)
\end{aligned}
\end{equation*}

\subsection{A sextic two-term identity}
 Using Lemma \ref{previously10p4}, we set $k=\frac{1}{2}\sec \frac{2\pi}{7}$, and this allows the application of Watson's formula 
 for $\operatorname{Li}_2(k) - \operatorname{Li}_2(k^2)$, so that 
\begin{equation*}
 \frac{-2k^2 + k - \sqrt{-3k^2 + 2k + 1} + 1}{2 - 2k^2}=r_2, \frac{-2k^2 + k + \sqrt{-3k^2 + 2k + 1} + 1}{2 - 2k^2}=r_1, 
\end{equation*}
 where $r_2 = -0.42812..., r_1=1.87316...$ Or exactly, $x^6 - 4 x^5 + 5 x^4 - 2 x^3 + x^2 - x - 1 =0$. 
 Using Roger's L-function, we have that 
\begin{equation*}
\frac{\pi^2}{7} = 
 \operatorname{Li}_2\left(\frac{1}{r_1}\right) + \operatorname{Li}_2\left(\frac{r_2}{r_2 - 1}\right) + 
 \frac{1}{2} \ln\left(\frac{1}{r_1}\right) \ln\left(1 - \frac{1}{r_1}\right) + \frac{1}{2} \ln\left(\frac{r_2}{r_2 - 
 1}\right) \ln\left(\frac{1}{1 - r_2}\right). 
\end{equation*}
 Now, if we were to set $x = \frac{1}{1-y}$, then we would yield Bytsko's two-term sextic dilogarithm identity \cite{Bytsko1999}, 
 which is composed of the two real roots of $y^6 - 7 y^5 + 19 y^4 - 28 y^3 + 20 y^2 - 7 y + 1=0$. He asks if it is 
 related to the Watson identities, as his result was experimentally derived. We can now answer in the affirmative.

\subsection{A two-term identity over $\mathbb{Q}(\sqrt{2})$}
 Bytsko \cite{Bytsko1999} also conjectured a two-term identity over $\mathbb{Q}(\sqrt{2})$ through asymptotic analysis. As 
 above, Bytsko noted that a proof is lacking. It can be proven by iterating Lemma \ref{previously10p4} again via Roger's 
 five-term identity. We thus have a new dilogarithm identity composed of $$ z(k) = \frac{k^2 + (k + 1) \sqrt{-3k^2 + 2k + 1} - 
 2k - 1}{2(k^2 - k - 1)} $$ and $$ w(k) = \frac{k^2 - (k + 1) \sqrt{-3k^2 + 2k + 1} - 2k - 1}{2(k^2 - k - 1)}. $$ Now, we have to 
 show that for $\operatorname{Li}_2(k) - \operatorname{Li}_2(k^2) - \operatorname{Li}_2\left( \frac{k^3}{(k - 1)(k + 1)^2} \right)$ the 
 dilogarithm parts cancel, in which we let $k=2^{-1/2}$. Substituting into $z(k)$, we have $\frac{1}{2} \left( 3 - \sqrt{2} - \sqrt{2\sqrt{2} - 
 1} \right) $, which is a root of $x^4 - 6 x^3 + 13 x^2 - 10 x + 1=0$. This agrees with Bytsko's result. 

\subsection{A ladder related to $\rm{Cl}_2\big(\tfrac{\pi}3\big)$}
 The Gieseking constant may be defined as the singular value of the Clausen function $\rm{Cl}_2\big(\tfrac{\pi}3\big)$. Using Lemma 
 \ref{previously10p3}, we solve 
\begin{equation}\label{previously13p7}
 \frac{(u-1) \left(2 u^3 - 4 u^2 + 3 u - \sqrt{u(4 - 3 u)}\right)}{2 \left(u^4 - 3 u^3 + 4 u^2 - 2 u + 1\right)} = \frac{1}{2} 
\end{equation}
 We also have $\sqrt{u (4 - 3u)} = \frac{2u^3 - 5u^2 + 2u + 2}{2 - u}$. And noting that $\frac{2 u^3 - 6 u^2 + 
 4 u + 1}{(u - 2)(u - 1)^2} = e^{\pi i / 3}$ is equivalent to \eqref{previously13p7}. 
 Thus, where $r_1=-0.141... - 0.215.. i$ or exactly as a root of $ x^4 - 4 x^3 + 14 x^2 + 4 x + 1=0$: 
\begin{equation*}
2 \space \Im\operatorname{Li}_2(r_1) - \frac{1}{2} \space\Im\operatorname{Li}_2(r_1^2) + \Im\operatorname{Li}_2\left(e^{i\pi/3}\right) = -2 \arg(1 - r_1) \ln|r_1| + \arg(1 - r_1^2) \ln|r_1|
\end{equation*}
 Moreover, we have that 
\begin{equation*}
4 \space\Im\operatorname{Li}_2(r_2) - 
 \frac{1}{2} \space\Im\operatorname{Li}_2(r_2^2) - 
 \frac{11}{6}\space \Im\operatorname{Li}_2\left(e^{i\pi/3}\right) = -4 \arg(1 - r_2) \ln|r_2| + \arg(1 - r_2^2) \ln|r_2|
\end{equation*}
Where $r_2=0.34861... + 0.20127... i$ or exactly as a root of $3 x^4 - 15 x^3 + 28 x^2 - 15 x + 3 = 0$. Technically, $\rm{Cl}_2\big(\tfrac{\pi}3\big)$ is a polylogarithmic value, so it is not a strict ladder. 

\subsection{A modified Watson $\frac{\pi}{7}$-dilogarithm identity}
 Bytsko \cite{Bytsko1999} also found a variant of Watson's $\frac{\pi}{7}$-identities. Bytsko proved this variant through elementary 
 manipulations from the classic Watson $\frac{\pi}{7}$-identities. It can also be proved using Lemma \ref{previously10p5}. 
 In this direction, we begin by setting 
\begin{equation*}
 \frac{-u^3 + 5u^2 - 5u}{2(u - 1)^2} = \frac{1}{2}. 
\end{equation*}
 The solution is $u = 2 - \frac{1}{2} \sec\left(\frac{3\pi}{7}\right)$, which gives the modified 
 two-term identity, through an appropriate substitution. 

\subsection{8th-degree single-term identity}
 Using Lemma \ref{previously10p5}, we solve (letting $u=1-y$) the equality such that 
\begin{equation*}
 \frac{y^3 + 2 y^2 + \sqrt{-3 y^2 + 2 y + 1} (-y^2 + y + 1) - 2 y - 1}{2 y^2} = \frac{z}{z - 1}, \quad \frac{1}{y(y - 1)} = z. 
\end{equation*}
 The single-term identity is as follows:
\begin{equation*}
\begin{aligned}
2 \operatorname{Li}_2\left(\frac{y}{y - 1}\right)
+ \frac{\pi^2}{3}
+ \frac{3}{2} \ln^2\left(\frac{w}{y}\right)
- \ln\left(\frac{w}{y^2}\right) \ln\left(-\frac{w}{y^2} + \frac{y + 1}{y^2}\right)\\
+ \frac{1}{2} \ln^2\left(\frac{1}{y} - 1\right)
+ \frac{1}{2} \ln^2\left((1 - y) y\right)
- \frac{1}{2} \ln^2\left(\frac{y^2 - y}{y^2 - y - 1}\right) = 0
\end{aligned}
\end{equation*}
 Where $y\approx1.3945 - 0.6970 i$ is a solution of $y^8 - 3 y^7 + y^6 + 4 y^5 - 6 y^3 + 3 y + 
 1=0$ and $w=\approx 1/2 - 1.14411 i$, which is a root of $w^8 - 4w^7 + 7w^6 - 7w^5 + 4w^4 - 
 w^3 + 2w^2 - 2w + 1 = 0$.

 It is only in exceptional cases that $\mathbb{Q}$-linear and two-term combinations of dilogarithm values with real and algebraic 
 arguments admit a closed-form evaluation, with reference to the Watson identities of this form, together with the cases 
 considered by Lewin corresponding to $v^a+v^b=1$. The sextic case is even more exceptional, and we have found only a single 
 instance of this case, through what may be regarded as a fortuitous application of one of the Watson identities. There is no 
 general method to determine all possible two-term dilogarithm evaluations. Informally, each such evaluation may be thought of as 
 requiring a unique manipulation of elementary identities. 

 We have demonstrated that all of the two-term dilogarithm identities due to Bytsko \cite{Bytsko1999}, as well 
 as the $\frac{\pi}{9}$-identities, can be proven through various manipulations and transformations of the 
 integrals in \eqref{w1w2forCampbell} in conjunction with applications of the radius method. 

\section{New Ladders}
 We derive a three-term ladder identity, given in \eqref{previously8p8} below, with arguments in $\mathbb{Q}(\sqrt{-7})$, 
 and this ladder relation appears to be new. We also introduce a derivation of the Lewin-Loxtin ladders using an elementary 
 analytic approach, which, to the best of our knowledge, is the first derivation of its kind. But finding new ladder relations is difficult, as 
 computer searches and other methods have exhausted the known options. 

\subsection{A three-term ladder relation}\label{subsection3ladder}
 We can rewrite Theorem \ref{previouslymaster} so that the three-term dilogarithm combination 
\begin{multline*} 
 \operatorname{Li}_2\left(\frac{2u^2 - 3u - \sqrt{u(4-3u)}}{2(u - 1)u}\right) - \\ 
 \operatorname{Li}_2\left(\frac{2u^2 - 3u - \sqrt{u(4 - 3u)}}{2(u - 1)^2}\right) 
 + \operatorname{Li}_2\left(\frac{u}{u - 1}\right) 
\end{multline*}
 admits the closed form 
 $$ - \frac{1}{2} \ln^2\left(\frac{2u^2 - 3u + \sqrt{u(4 - 3u)}}{-2(u - 1)^2}\right) - \frac{\pi^2}{6}
 + \frac{3}{2} \ln^2\left(1 - \frac{u - \sqrt{u(4 - 3u)}}{2(u - 1)}\right). $$
 We proceed to apply the substitutions 
\begin{align*}
\frac{2u^2 - 3u - \sqrt{-u(3u - 4)}}{2(u - 1)u} & = \left(\frac{u}{u - 1}\right)^m,\\
\frac{2u^2 - 3u + \sqrt{-u(3u - 4)}}{2(u - 1)^2} & = \left(\frac{u}{u - 1}\right)^{-m}, \ \text{and} \\
\frac{2u^2 - 3u - \sqrt{-u(3u - 4)}}{2(u - 1)^2} & = \left(\frac{u}{u - 1}\right)^{m + 1}. 
\end{align*}
 Substituting $u = \frac{h}{h-1}$, we have the ladder identity 
\begin{equation*}
\begin{aligned}
- \mathrm{Li}_2(h^{-m}) - \mathrm{Li}_2(h^{m+1}) + \mathrm{Li}_2(h) = \frac{3}{2} \ln^2\left( \frac{1 - h^{-m}}{1 - h} \right)
\end{aligned}
\end{equation*}
 for arbitrary powers, where $(h^m + h - 3) h^{m + 1} + 1 = 0$. 

\begin{example}
 For $m=-4$ we have that $h^5 + 2 h^4 + 4 h^3 + 3 h^2 + 2 h + 1=0$. 
 This can be proved by applying Roger's identity and then Landen's identity. 
 A similar approach can be applied to many of Lewin's ladder identities. 
\end{example}

\subsection{Five-term ladder relations, part 1}\label{section5term1}
 In this section we derive identities are are not as trivially verifiable as the above example. Some well-known dilogarithm bases include 
 the basis such that $$u^6-3u^4-4u^3-3u^2+1=0$$ due to Lewin \cite{Lewin1991text}, the basis such that $$u^6-u^5 - 
 u^3-u+1=0$$ due to Abouzahra and Lewin \cite{AbouzahraLewin1985}. These bases motivate the following dilogarithm identities. 
 Since the sextic polynomial in Theorem \ref{firstformatladder} is symmetric, we obtain a reduction to a cubic polynomial identity 
 with a fundamental discriminant of level $29$. 

\begin{theorem}\label{firstformatladder}
 For the unique root $h$ in the interval $(0, 1)$ of $x^6-x^5-x^4-6 x^3-x^2-x+1 = 0$, the valid ladder relation 
 $$ \operatorname{Li}_{2}\left( h^{6} \right) + 3 \operatorname{Li}_{2}\left( h^{4} \right) 
 - 8 \operatorname{Li}_{2}\left( h^{3} \right) - 5 \operatorname{Li}_{2}\left( h^2 \right) 
 - 8 \operatorname{Li}_{2}(h) - 4 \ln^2(h) = - 5 \zeta(2) $$ 
 holds. 
\end{theorem}

\begin{proof}
 From the the $A$- and $D$-expressions in Theorem \ref{maintheorem}, by applying the four-term relation 
 in \eqref{wellknown4term} to the $K$-expression in Theorem \ref{maintheorem}, this gives us an evaluation for 
\begin{multline*}
 \text{ {\footnotesize $ \frac{1}{2} \operatorname{Li}_2\left(\left(\frac{u^2 + \sqrt{4 - 3 u^2} - 2}{u(u + 
 1)}\right)^2\right) - 4 \operatorname{Li}_2\left(\frac{u^2 + \sqrt{4 - 3 u^2} - 2}{u(u + 1)}\right) - 
 4 \operatorname{Li}_2\left(\frac{1 - u}{u + 1}\right) $} } \\ 
 \text{ { \footnotesize$ - \frac{1}{2} \operatorname{Li}_2\left(\left(\frac{u^2 + \sqrt{4 - 3 u^2} - 2}{(u - 1) u}\right)^2\right) + 
 4 \operatorname{Li}_2\left(\frac{u^2 + \sqrt{4 - 3 u^2} - 2}{(u - 1) u}\right) 
 + \frac{3}{2} \operatorname{Li}_2\left(\left(\frac{1 - u}{u + 1}\right)^2\right) 
 $}}. 
\end{multline*}
 From this evaluation, by mimicking how the three-term $\operatorname{Li}_{2}$-relation in Section \ref{subsection3ladder} was used to 
 construct a ladder, we set $$ \frac{\left(u^2 + \sqrt{4 - 3u^2} - 2\right)^2}{u^2 (u + 1)^2} = \frac{u + 1}{1 - u} \ \ \ \text{and} \ \ \ 
 \frac{\left(u^2 + \sqrt{4 - 3u^2} - 2\right)^2}{(u - 1)^2 u^2} = \frac{(u + 1)^3}{(1 - u)^3}, $$ yielding the same solutions. So, by setting 
 $h = \sqrt{\frac{1 + u}{1 - u}}$, we obtain the root of $x^6 - x^5 - x^4 - 6 x^3 - x^2 - x + 1 =0$ specified. This gives us that 
\begin{multline*}
 \frac{5}{4} \ln(h^2)\ln(1 - h^2) + \frac{4}{2} \ln(h)\ln(1 - h) - \frac{3}{4} \ln(h^4)\ln(1 - h^4) + \frac{4}{2} \ln(h^3)\ln(1 - 
 h^3) - \\ 
 \frac{1}{4} \ln(h^6)\ln(1 - h^6) + \frac{5}{2} \operatorname{Li}_2(h^2) + 
 4 \operatorname{Li}_2(h) - \frac{3}{2} \operatorname{Li}_2(h^4) + 
 4 \operatorname{Li}_2(h^3) - \frac{1}{2} \operatorname{Li}_2(h^6) = \frac{5\pi^2}{12}, 
\end{multline*}
 and, after much simplificatoin, this can be shown to be equivalent to the desired result. 
\end{proof}

\subsection{A six-term ladder relation}
 An analogue of our approach from Section \ref{section5term1} leads us to construct and prove a new six-term ladder relation, 
 as below. 

\begin{theorem}\label{theoremnew6term}
 Let $h$ denote the unique root in the interval $(0, 1)$ of $x^6-3 x^5+5 x^4-8 x^3+5 x^2-3 x+1 = 0$. Then the ladder relation 
\begin{multline*}
 \operatorname{Li}_2\left(h^8\right) + 3 \operatorname{Li}_2\left(h^6\right) - 
 8 \operatorname{Li}_2\left(h^4\right) - 8 \operatorname{Li}_2\left(h^3\right) + 
 3 \operatorname{Li}_2\left(h^2\right) - 8 \operatorname{Li}_2(h)-6 \ln ^2(h) = \\ 
 -5 \zeta(2) 
\end{multline*}
 holds. 
\end{theorem}

\begin{proof}
 We exploit the property whereby the following equalities 
\begin{equation}\label{commonsolutionfor6}
 \frac{u^2 - \sqrt{4 - 3u^2} - 2}{u(u - 1)} = \left( \frac{u + 1}{1 - u} \right)^4 
 \ \ \ \text{and} \ \ \ 
 \frac{u^2 + \sqrt{4 - 3u^2} - 2}{u(u + 1)} = \left( \frac{1 - u}{1 + u} \right)^4 
\end{equation}
 have a common solution. In this direction, the six-term dilogarithm equation 
 involved in our proof of Theorem \ref{firstformatladder} gives us a corresponding evaluation for 
\begin{multline*}
 \operatorname{Li}_2\left( \frac{u^2 + \sqrt{4 - 3u^2} - 2}{-u(u + 1)} \right) 
 - 3 \operatorname{Li}_2\left( \frac{u^2 + \sqrt{4 - 3u^2} - 2}{u(u + 1)} \right) 
 - 4 \operatorname{Li}_2\left( \frac{1 - u}{u + 1} \right) \\ 
 + \operatorname{Li}_2\left( \frac{u^2 - \sqrt{4 - 3u^2} - 2}{-u(u + 1)} \right) 
 - 3 \operatorname{Li}_2\left( \frac{u^2 - \sqrt{4 - 3u^2} - 2}{u(u + 1)} \right) 
 + \frac{3}{2} \operatorname{Li}_2\left( \left( \frac{1 - u}{u + 1} \right)^2 \right), 
\end{multline*}
 yielding, from the common solutions for \eqref{commonsolutionfor6}, the equality of 
 $$ - \frac{1}{2} \operatorname{Li}_2(h^8) 
 - \frac{3}{2} \operatorname{Li}_2(h^6) 
 + 4 \operatorname{Li}_2(h^4) 
 + 4 \operatorname{Li}_2(h^3) 
 -\frac{3}{2} \operatorname{Li}_2(h^2) 
 + 4 \operatorname{Li}_2(h) $$ and 
\begin{multline*}
 \text{ { \footnotesize $ \frac{3}{4} \ln(h^2) \ln(1 - h^2)
 - 2 \ln(h) \ln(1 - h)
 - 2 \ln(h^4) \ln(1 - h^4) 
 - 2 \ln(h^3) \ln(1 - h^3) + $ } } \\ 
 \text{{ \footnotesize $ \frac{3}{4} \ln(h^6) \ln(1 - h^6) + \frac{1}{4} \ln(h^8) \ln(1 - h^8) + \frac{5\pi^2}{12} $ }}
\end{multline*}
 for $h$ as specified, and this can be shown
 to provide an equivalent version of the desired result. 
\end{proof}

\begin{remark}
 Since the sexic polynomial in Theorem \ref{theoremnew6term} is symmetric, it can be reduced to a cubic case that has a fundamental 
 discriminant of 13. This is the only cubic ladder with arguments in $\mathbb{Q}(\sqrt{13})$, to the best of our knowledge. 
\end{remark}

 By analogy with the ladder constructions due to Lewin and Campbell, is is possible to generate ladders for arbitrary powers with the 
 base equation 
 \begin{equation*}
(h - 1) h^{2 m-1} = (h (h + 6) + 1) h^{m-1} +h - 1 
\end{equation*}
 The $m=4$ case recovers Theorem \ref{theoremnew6term} and the $m=3$ case recovers Theorem \ref{firstformatladder}. 
 This base equation appears new and not seen the literature consulted and does not readily follow from the Lewin identities or others. 

\subsection{Using the ``radius method'' to derive new ladders}
 The radius method, as described earlier to derive analytic proofs of the trio of $\frac{\pi}{9}$ identities, can be applied here as well. 
 We apply the five-term dilogarithm identity involved in the proof of Theorem \ref{theoremnew6term}, setting 
\begin{equation*}
\frac{\sqrt{4-3u^2}}{ u^2 - 2} = -i \sqrt{3}
\end{equation*}
\\
The radius is $r = \sqrt{\frac{u - 1}{u + 1}} \quad \text{proof:} \quad r^2 = \frac{(u - 1)^2}{u^2 - 1} = \frac{u - 1}{u + 1} $. 
 Explicitly, we have that 
\begin{equation}\label{radiusfornewladder1}
 r=\frac{1}{4} \left( -1 + \sqrt{33} - \sqrt{2 \left(9 - \sqrt{33} \right)} \right) 
\end{equation}
 Consequently, we have that 
\begin{equation*}
\begin{aligned}
\operatorname{Li}_2\left( \frac{u^2 - \sqrt{4 - 3 u^2} - 2}{-(u + 1) u} \right)
+ \operatorname{Li}_2\left( \frac{u^2 + \sqrt{4 - 3 u^2} - 2}{-(u + 1) u} \right)
= -\operatorname{Li}_2(r) + \frac{1}{3} \operatorname{Li}_2(r^3)\\
\operatorname{Li}_2\left( \frac{u^2 - \sqrt{4 - 3 u^2} - 2}{(u + 1) u} \right)
+ \operatorname{Li}_2\left( \frac{u^2 + \sqrt{4 - 3 u^2} - 2}{(u + 1) u} \right)
= -\operatorname{Li}_2(-r) + \frac{1}{3} \operatorname{Li}_2(-r^3)
\end{aligned}
\end{equation*}
 The resultant ladder relation is such that 
\begin{multline*}
 -\frac{1}{2} \operatorname{Li}_2(r^6) - \frac{1}{2} \operatorname{Li}_2(r^4) + \frac{4}{3} \operatorname{Li}_2(r^3) + 
 \frac{11}{2} \operatorname{Li}_2(r^2) - 4 \operatorname{Li}_2(r) + 
 \frac{11}{4} \ln(r^2)\ln(1 - r^2) - \\ 
 2 \ln(r)\ln(1 - r) - \frac{1}{4} \ln(r^4)\ln(1 - r^4) + \frac{2}{3} \ln(r^3)\ln(1 - r^3) - \frac{1}{4} \ln(r^6)\ln(1 - r^6) = -\frac{\pi^2}{36} 
\end{multline*}
 Using a mapping of $u^4 - 5 u^2 + \frac{16}{3} = 0$, where $u = \frac{1+r^2}{1 - r^2}$, we can determine that 
 all possible radii $r$ for ladders derived in the above manner are positive solutions of $(r^4 - r^3 - 6 r^2 - r + 
 1) (r^4 + r^3 - 6 r^2 + r + 1) = 0$, so there are four possible ladders, corresponding to $r_1$, $r_2$, $r_3$, and $r_4$. 
 These ladders may be thought of as being analogues of the trio of ladders related to the value $\frac{\pi}{9}$. Two of these are 
 redundant, given that $r_1=r_3^{-1}$ and $r_2=r_4^{-1}$. 

 Setting $\frac{\sqrt{4 - 3u^2}}{2 - u^2} = \sqrt{3} \, i$, gives the second radius, with 
\begin{equation}\label{radiusfornewladder2}
 r_2 = \frac{1}{4} \left(1 + \sqrt{33} - 
 \sqrt{2 \left(9 + \sqrt{33} \right)} \right)
\end{equation}
 The associated ladder relation is such that 
\begin{multline*}
 \frac{\operatorname{Li}_2(r_2^6)}{6} - \frac{\operatorname{Li}_2(r_2^4)}{2} - \frac{4 \operatorname{Li}_2(r_2^3)}{3} + \frac{7 
 \operatorname{Li}_2(r_2^2)}{2} + 4 \operatorname{Li}_2(r_2) 
 + \frac{7 \ln(r_2^2) \ln(1 - r_2^2) }{4} \\ 
 + 2 \ln(r_2) \ln(1 - r_2) 
 - \frac{\ln(r_2^4) \ln(1 - r_2^4) }{4} - \frac{2 \ln(r_2^3) \ln(1 - r_2^3) }{3} 
 + \frac{ \ln(r_2^6) \ln(1 - r_2^6) }{12} = \frac{11\pi^2}{36} 
 \end{multline*}
 We emphasize that our ladder relations involving the value $r$ in \eqref{radiusfornewladder1} and the value $r_{2}$ in 
 \eqref{radiusfornewladder2} appear to be new and not in any of the literature consulted. Gordon and Mcintosh 
 \cite{GordonMcIntosh1997} note the difficulty about finding dilogarithm ladders that are not of the typical Lewin format. Indeed, 
 this may be seen as requiring innovative approaches, and we encourage further applications of our radius method. 
 
\subsection*{Acknowledgements}
 The author thanks John Maxwell Campbell for help with the exposition in this paper and for help with the typesetting and formatting 
 for this paper.

 \

Cetin Hakimoglu-Brown

Berkeley, CA

{\tt mathemails@proton.me}

\end{document}